\documentclass[11pt]{amsart}
\usepackage{amsmath}
\usepackage{amssymb}
\usepackage{amsthm}
\usepackage{amscd}
\usepackage{fontenc}
\usepackage{url}
\usepackage[all]{xy}

\numberwithin{equation}{section}

\newcounter{AbcT}

\newtheorem {Theorem}    {Theorem}[section]

\newtheorem*{remarkstar}{\textit{Remark}}

\newtheorem {Example}    {Example}

\newtheorem {Lemma}      [Theorem]    {Lemma}
\newtheorem {Corollary}   [Theorem] {Corollary}
\newtheorem {Proposition}[Theorem]    {Proposition}
\newtheorem {Claim}      [Theorem]    {Claim}

\theoremstyle{remark}
\newtheorem {Remark}		 [Theorem]    {\bf{Remark}}

\newtheorem {Definition} [Theorem]    {\bf{Definition}}

\newcounter{DM@bibnum}

\newcommand {\R} {{\mathbb R}}

\newcommand {\Z} {{\mathbb Z}}

\newcommand{\SL}{\operatorname {SL}}
\newcommand{\EL}{\operatorname {EL}}
 \newcommand{\rr}{\mathbb R}
  \newcommand{\cc}{\mathbb C}
 \newcommand{\zz}{\mathbb Z}

 \newcommand{\sltwo}{\mathfrak{sl}_2}

  \newcommand{\norm}[1]{\left\Vert#1\right\Vert}

\newcommand{\la}{\langle}
\newcommand{\ra}{\rangle}

\def\orth{{\rm orth}}

\def\log{{\rm log\,}}

\def\Aut{{\rm Aut\,}}

\def\NSL_2{{\mathcal N SL_2}}

\def\eps{\varepsilon}

\def\lam{\lambda}            
                
\def\phi{\varphi}

\def\sgal{\tilde s}

\def\hbar{\bar h}

\def\dbC{{\mathbb C}}

\def\dbE{{\mathbb E}}

\def\dbG{{\mathbb G}}

\def\dbL{{\mathbb L}}

\def\dbN{{\mathbb N}}

\def\dbQ{{\mathbb Q}}
\def\dbR{{\mathbb R}}

\def\dbZ{{\mathbb Z}}

\begin{document}

\title{Property (T) for Kac-Moody groups over rings}
\author{Mikhail Ershov}
\address{University of Virginia}
\email{ershov@virginia.edu}
\author{Ashley Rall}
\address{University of Virginia}
\email{ashley.rall@email.virginia.edu}
\address{University of Virginia}
\email{zz2d@virginia.edu}

\thanks{The work of the first author was partially supported by the NSF grant DMS-1201452
and the Sloan Research Fellowship grant BR 2011-105.}

\keywords{Kac-Moody groups, property (T)}

\maketitle
\centerline{ with an appendix by }
\vskip .1cm
\centerline{\sc Zezhou Zhang}
\vskip .7cm
\centerline{\it To Efim Zelmanov on the occasion of his 60th
birthday}

\begin{abstract}
Let $R$ be a finitely generated commutative ring with $1$, let $A$ be an indecomposable $2$-spherical generalized Cartan matrix of size at least $2$ and $M=M(A)$ the largest absolute value of a non-diagonal entry of $A$.
We prove that there exists an integer $n=n(A)$ such that the Kac-Moody group $\dbG_A(R)$ has property $(T)$ whenever $R$ has no proper ideals of index less than $n$ and all positive integers less than or equal to
$M$ are invertible in $R$.
\end{abstract}

\section{Introduction}

Kac-Moody groups can be thought of as infinite-dimensional analogues of Chevalley groups or algebraic groups.
Over the last two decades they have attracted a lot of attention from mathematicians working
in many different areas. Kac-Moody groups were shown to have a deep and interesting structure theory; at the same
time they provided an excellent source of test examples for various conjectures and helped settle open problems 
in geometric group theory, Lie theory and related areas (see, e.g., \cite{CR2,Re} and references therein). Most of the results obtained so far deal with Kac-Moody groups over fields, with the case of finite fields proving to be particularly interesting. At the same time,
the subject of Kac-Moody groups over rings remains a largely uncharted territory, so much so that even
the ``right'' definition over non-fields is yet to be agreed on. 

Kac-Moody groups over rings in the sense of this paper are defined very explicitly by generators and relations. This is the definition used, for instance, in recent papers of Allcock~\cite{Al,Al2}; see \S~2 for a brief discussion
of other possible definitions and connections between them.  
Given a generalized Cartan matrix $A$, let $\Phi=\Phi(A)$ be the associated system of real roots. For an arbitrary commutative ring $R$ (with $1$), the corresponding simply-connected Kac-Moody group $\dbG_A(R)$ is the group generated by the root subgroups $\{X_{\alpha}\}_{\alpha\in \Phi}$, each of which is isomorphic to the additive group of $R$, modulo certain Steinberg-type relations.\footnote{In this paper we will only discuss simply-connected Kac-Moody groups, so for brevity we
will restrict all the definitions to the simply-connected  case.} In particular, if $A$ is a matrix of finite
(spherical) type and $R$ is a field,  $\dbG_A(R)$ is the corresponding Chevalley group with its standard Steinberg presentation.
It is easy to see that the obtained correspondence $R\mapsto \dbG_A(R)$ is functorial and for any epimorphisms
of rings $R\to S$ the corresponding map $\dbG_A(R)\to \dbG_A(S)$ is also an epimorphism.

The main goal of this paper is to establish a sufficient condition for Kac-Moody groups over rings
to have Kazhdan's property $(T)$. The first results on property $(T)$ for (non-spherical and non-affine) Kac-Moody groups are due to Dymara and Januszkiewicz \cite{DJ} who proved that the group $\dbG_A(F)$ has property $(T)$ for any
indecomposable $2$-spherical generalized Cartan matrix $A$ and any finite field $F$ satisfying $|F|>1764^{d(A)}$, where $d(A)$ denotes the size of $A$. In \cite{Op1}, Oppenheim obtained a quantitative improvement of this result, replacing exponential bound in $d(A)$ by a polynomial (in fact, quadratic) one; see also \cite{Op2} where a generalization of property~(T) dealing with affine isometric actions on Banach spaces is established under similar restrictions on $F$. It does not seem possible to extend those proofs to groups over non-fields since both arguments make essential use of the action of Kac-Moody groups on the associated  buildings, which can only be constructed in the case of fields. 

In \cite{EJ}, an algebraic counterpart of the method from \cite{DJ} was used to establish property $(T)$ for 
``unipotent subgroups'' of Kac-Moody groups over finite rings. The ``positive unipotent subgroup'' $\dbG_A^+(R)$
of the Kac-Moody group $\dbG_A(R)$ is defined to be the subgroup of $\dbG_A(R)$ generated by positive root subgroups.
In \cite{EJ}, it was shown that the group $\dbG_A^+(R)$ has property $(T)$ for any indecomposable $2$-spherical $d\times d$ generalized Cartan matrix $A$ with simply-laced Dynkin diagram and any finite commutative ring $R$ with $1$ which has no proper ideals of index less than $(d-1)^2$. The abelianization of the group $\dbG_A^+(R)$ is isomorphic to a direct sum of several copies of $(R,+)$, whence $\dbG_A^+(R)$ cannot possibly have property $(T)$ if $R$ is infinite. However, it turns out that the techniques from \cite{EJ} can be used to prove property $(T)$ for the full Kac-Moody group $\dbG_A(R)$ over an arbitrary finitely generated commutative ring $R$ under a similar
restriction on indices of ideals in $R$:

\begin{Theorem}
\label{thm:main}
Let $A=(a_{ij})$ be an indecomposable 2-spherical generalized Cartan matrix of size $d\geq 2$, let $M=M(A)=\max\{|a_{ij}|: i\neq j\}=\max\{-a_{ij}: i\neq j\}$ be the largest absolute value of a non-diagonal
entry of $A$ (the assumptions on $A$ imply that $1\leq M\leq 3$).
Define the integer $n=n(A)$ as follows:
\begin{itemize}
\item[(i)] $n=(2d-2)^2$ if $M=1$ (equivalently, the Dynkin diagram of $A$ is simply-laced)
\item[(ii)] $n=3(2d-2)^4$ if $M=2$ (equivalently, the Dynkin diagram of $A$ has a double edge, but no triple edges)
\item[(iii)] $n=188(2d-2)^{16}$ if $M=3$ (equivalently, the Dynkin diagram of $A$ has a triple edge)
\end{itemize}
Let $R$ be any finitely generated commutative ring with $1$ which does not have proper ideals of index less than $n$ and such that every positive integer $\leq M$ is invertible in $R$. Then the Kac-Moody group $\dbG_A(R)$ has Kazhdan's property $(T)$.
\end{Theorem}

As a simple corollary of this theorem, we deduce a uniform bound for Kazhdan constants
(with respect to generating sets of bounded size) for Kac-Moody groups of a fixed indecomposable $2$-spherical
type over finite fields of sufficiently large characteristic:

\begin{Corollary}  Let $A$ be an indecomposable $2$-spherical generalized Cartan matrix of size $d\geq 2$, and let $n=n(A)$ be defined as in Theorem~\ref{thm:main}. There exist an integer $k=k(A)$ and a real number $\eps=\eps(A)>0$ 
with the following property: for any finite field $F$ with 
${\rm char}(F)> n$ there exists a finite generating set $S(F)$ of $\dbG_A(F)$ with $|S(F)|\leq k$
such that $\kappa(\dbG_A(F),S(F))\geq \eps$ where $\kappa(\dbG_A(F),S(F))$ is the Kazhdan constant of $\dbG_A(F)$ with respect
to $S(F)$. 
\end{Corollary}

We do not know if the above result remains true if a lower bound on characteristic of $F$ is replaced
by a lower bound on its size, even if we assume that $A$ is simply-laced. To prove the corollary observe that the ring $R_n=\mathbb Z[\frac{1}{n!},t]$ (the ring of polynomials in one variable over $\mathbb Z[\frac{1}{n!}]$) has no ideals of index less than $n$ and surjects onto any finite field of characteristic larger than $n$. Therefore, for any finite field $F$ with ${\rm char}(F)>n$ there exists a natural epimorphism $\pi: \dbG_A(R_n)\to \dbG_A(F)$. If $S$ is any finite generating set for $\dbG_A(R_n)$ and $S(F)$ is its image in $\dbG_A(F)$, we get
$\kappa(\dbG_A(F),S(F))\geq \kappa(\dbG_A(R_n),S)$, and $\kappa(\dbG_A(R_n),S)>0$ since $\dbG_A(R_n)$
has property $(T)$ by Theorem~\ref{thm:main}.
\vskip .1cm
Another interesting result on property $(T)$ for Kac-Moody groups was obtained by Hartnick and K\"ohl~\cite{HK} who proved
that for any local field $F$ and any indecomposable 2-spherical $d\times d$ generalized Cartan matrix $A$, with $d\geq 2$, 
the group $\dbG_A(F)$ has property $(T)$ when considered as a topological group with the Kac-Peterson topology. We will give an alternative proof
of this theorem in \S~5.
\vskip .1cm

{\bf Organization:} In \S~2 we recall basic properties of Kac-Moody root systems and define Kac-Moody groups over commutative rings. In \S~3 we collect background information on property $(T)$. In \S~4 we establish some auxiliary results on orthogonality constants in Chevalley groups of rank $2$
(see \S~3 for the definition of orthogonality constants). In \S~5 we prove Theorem~\ref{thm:main} as well as its variation dealing with ``pseudo-parabolic'' subgroups and give a new proof of the theorem of Hartnick and K\"ohl mentioned above. 
\vskip .1cm

{\bf Convention:} All rings considered in this paper are assumed to be unital.

\vskip .1cm
{\bf Acknowledgments.} We are grateful to Pierre-Emmanuel Caprace for useful feedback and for suggesting a much simpler proof of Theorem~\ref{KMreals}. We would also like to thank Zezhou Zhang for suggesting a stronger version of Theorem~\ref{thm:pseudoparabolic} and providing useful references.

\section{Preliminaries on Kac-Moody groups}
In this section we recall basic properties of Kac-Moody root systems, define Kac-Moody groups over rings and state some basic
facts about them. Note that Kac-Moody groups can be defined without an explicit reference to Kac-Moody Lie algebras even though Lie algebras are needed to establish some key properties of Kac-Moody groups. For more details we refer the reader to the book of 
Kac~\cite{Kac} and recent paper of Allcock~\cite{Al}.

\subsection{Kac-Moody root systems} 

\vskip .05cm
Let $A=(a_{ij})$ be a {\it generalized Cartan matrix} (abbreviated below as GCM), that is, a square matrix satisfying the following conditions:
\begin{align*}
&\mbox{(a)}\quad a_{ij}\in{\dbZ} \mbox{ for all } i,j &
&\mbox{(b)}\quad a_{ii}=2 \mbox{ for all } i&\\
&\mbox{(c)}\quad a_{ij}\leq 0 \mbox{  if } i\neq j&
&\mbox{(d)}\quad a_{ij}=0 \iff a_{ji}=0.&
\end{align*}
For the rest of the section we fix a GCM $A$ with entries $a_{ij}$ and let $d$ denote its size.

Denote by $\Delta=\Delta(A)$ the root system of the complex Kac-Moody Lie algebra $\mathfrak g=\mathfrak g(A)$ associated to $A$ and by $\Phi=\Phi(A)$ the set of real roots in $\Delta$. We shall not use $\Delta$ or $\mathfrak g$ in this paper, so we will define $\Phi$ directly in terms of $A$.

Let $Q=\oplus_{i=1}^d \dbZ\alpha_i$ and $Q^{\vee}=\oplus_{i=1}^d \dbZ\alpha_i^{\vee}$ be free abelian groups of rank $d$
with bases $\Pi=\{\alpha_1,\ldots \alpha_d\}$ and $\Pi^{\vee}=\{\alpha_1^{\vee},\ldots \alpha_d^{\vee}\}$. Define the bilinear pairing
$\la \cdot,\cdot \ra:Q^{\vee}\times Q\to\dbZ$ by $\la \alpha_i^{\vee},\alpha_j\ra=a_{ij}$.
For each $1\leq i\leq d$ define the map $s_i\in \Aut(Q)$ by $s_i(x)=x-\la \alpha_i^{\vee},x\ra\alpha_i$; in particular
$s_i(\alpha_j)=\alpha_j -a_{ij}\alpha_i$.

Let $W=\la s_1,\ldots, s_d\ra$ be the subgroup of $\Aut(Q)$ generated by $\{s_i\}_{i=1}^d$, and define $\Phi=W(\Pi)$ to be the union of $W$-orbits of $\Pi$. It is not hard to check that the group $W$ is a Coxeter group; it is called the {\it Weyl group} of $A$. The elements of $\Pi$ are called {\it simple roots}, and elements of $\Phi$ are called {\it real roots};
we will refer to real roots just as {\it roots} in this paper since we will never deal with imaginary roots. As in the case of finite root systems, every root is a linear combination of simple roots with all coefficients non-negative or all coefficients
non-positive; roots are called positive and negative, accordingly. The sets of positive and negative roots will be denoted by $\Phi^+$ and $\Phi^-$, respectively. 

The Weyl group $W$ has unique action on $Q^{\vee}$ such that $\la w \alpha_i^{\vee},w \alpha_j\ra=\la \alpha_i^{\vee},\alpha_j\ra$.
For each root $\alpha\in \Phi$ define $\alpha^{\vee}\in Q^{\vee}$ as follows: choose $1\leq i\leq d$ and $w\in W$ such that
$\alpha=w\alpha_i$ and define $\alpha^{\vee}=w\alpha_i^{\vee}$; this definition does not depend on the representation of $\alpha$ as $w\alpha_i$. Define $s_{\alpha}\in \Aut(Q)$ by $s_{\alpha}(\beta)=\beta-\la \alpha^{\vee},\beta\ra \alpha$. It is easy to see that $s_{w\alpha}=ws_{\alpha}w^{-1}$ for all $\alpha\in\Phi, w\in W$; in particular
this implies that each $s_{\alpha}\in W$.

\begin{Definition} Given a subset $I$ of $\{1,\ldots,d\}$, we let $A_I$ be the $|I|\times |I|$ matrix $(a_{ij})_{i,j\in I}$. Matrices of the form $A_I$ will be called {\it submatrices of $A$}.
\end{Definition}

\begin{Definition} Let $A$ be a GCM.
\begin{itemize}
\item[(i)] $A$ is called {\it indecomposable} if there is no partition $\{1,\ldots,d\}=I\sqcup J$ with $I,J\neq\emptyset$
such that $a_{ij}= 0$ for all $i\in I,j\in J$.
\item[(ii)] $A$ is called {\it spherical} (or finite) if it is positive definite.
\item[(iii)] $A$ is called {\it affine} if $A$ is positive semi-definite and all of its proper indecomposable submatrices are spherical.
\item[(iv)] Given an integer $2\leq k\leq d$, the matrix $A$ is called {\it k-spherical} if for every $I\subseteq\{1,\ldots, k\}$
with $|I|=k$, the submatrix $A_I$ is spherical. This is equivalent to saying that for any such $I$ the subgroup
$\la s_i: i\in I\ra$ of $W$ is finite. 
\end{itemize}
\end{Definition}
It is easy to see that $A=(a_{ij})$ is {\it 2-spherical} if and only if $a_{ij}a_{ji}\leq 3$ for all $i\neq j$.

\begin{Definition}\rm A pair of roots $\alpha,\beta$ in $\Phi$ is called {\it prenilpotent} if there exist $w,w'\in W$
such that $w\alpha,w\beta\in \Phi^+$ and $w'\alpha,w'\beta\in \Phi^-$.
\end{Definition}

In the following proposition we collect some well-known properties of prenilpotent pairs which will be used in this paper.

\begin{Proposition} 
\label{prenil_basic}
Let $\alpha,\beta\in \Phi$ with $\beta\neq -\alpha$. The following hold:
\begin{itemize}
\item[(a)] $\{\alpha,\beta\}$ is prenilpotent if and only if the set $(\dbN\alpha+\dbN\beta)\cap \Phi$ is finite.
\item[(b)] The numbers $\la \alpha^{\vee},\beta\ra$ and $\la \beta^{\vee},\alpha\ra$ are both positive, both negative
or both zero.  
\item[(c)] If $\la \alpha^{\vee},\beta\ra\geq 0$, then $\{\alpha,\beta\}$ is prenilpotent, and moreover
$(\dbN\alpha+\dbN\beta)\cap \Phi\subseteq \{\alpha+\beta\}$, that is, $\alpha+\beta$ is the only possible root
in $\dbN\alpha+\dbN\beta$.
\item[(d)] Assume that $\la \alpha^{\vee},\beta\ra< 0$. Then the following are equivalent:
\begin{itemize}
\item[(i)] $\{\alpha,\beta\}$ is prenilpotent;
\item[(ii)] $\la s_{\alpha},s_{\beta}\ra$ is finite;
\item[(iii)] $\la \alpha^{\vee},\beta \ra \cdot \la \beta^{\vee},\alpha \ra\leq 3$.
\end{itemize}
\end{itemize}
\end{Proposition}
\begin{proof} (a) holds by \cite[Proposition~4.7]{KP2}. Note that in \cite{KP2} the result is proved under the additional
hypothesis that $\alpha,\beta\in\Phi^+$; however, (a) easily follows from this special case. Indeed, both conditions in (a)
do not change if we replace $\{\alpha,\beta\}$ by $\{w\alpha,w\beta\}$ for some $w\in W$. Since $\beta\neq -\alpha$,
it is easy to see that there exists $w\in W$ such that $w\alpha$ and $w\beta$ are both positive or both negative. In the
former case we are reduced to the situation in \cite{KP2}, and in the latter case we use the obvious fact that $\{\gamma,\delta\}$ is prenilpotent if and only if $\{-\gamma,-\delta\}$ is prenilpotent.

(b) holds, for example, by \cite[p.139]{KP1} (see the argument after Lemma~1.2).

(c) follows from (a) and \cite[Lemma~2.1(c)(ii)]{KP2} (using the same remark as in the proof of (a)).
Finally, (d) follows from the proof of  \cite[Proposition~4.7]{KP2}.
\end{proof}

Finally, recall the notion of the {\it Dynkin diagram} of $A$, which we will denote by $Dyn(A)$.
We define $Dyn(A)$ to be a graph with vertex set $\Pi=\{\alpha_1,\ldots, \alpha_d\}$,
where $\alpha_i$ and $\alpha_j$ (for $i\neq j$) are connected by $a_{ij}a_{ji}$ edges.
Given a subset $I$ of $\{1,\ldots, d\}$, we denote by $Dyn_I(A)$ the full subgraph
of $Dyn(A)$ on the vertex set $\{\alpha_i: i\in I\}$. We will refer to $Dyn_I(A)$
as a {\it Dynkin subdiagram}.

\subsection{Definition of Kac-Moody groups over rings}
In this subsection we define Kac-Moody groups over rings by certain presentations by generators and relators. Our definition is easily seen to be equivalent to the one in Allcock's paper~\cite{Al}. 

Let $A$ be a GCM and $R$ a commutative ring with $1$. Define the Kac-Moody group $\dbG_A(R)$ to be the group
with generators $\{x_{\alpha}(r): \alpha\in\Phi, r\in R\}$ subject to relations (R1)-(R7) below. In those relations
$r,u\in R$ and $\alpha,\beta\in\Phi$ are arbitrary unless a restriction is explicitly imposed.
The signs in the relations (R3) are not canonical and depend on the choice of a Chevalley basis for the Kac-Moody
Lie algebra associated to $A$ (see \cite[3.2]{Ti} for details). 

(R1) $x_{\alpha}(r+u)=x_{\alpha}(r)x_{\alpha}(u)$

(R2) If $\{\alpha,\beta\}$ is a prenilpotent pair, then
$$[x_{\alpha}(r),x_{\beta}(u)]=
\prod_{i,j\geq 1}x_{i\alpha+j\beta}(C_{ij\alpha\beta}r^i u^j)$$
\vskip -.2cm
\noindent
where
the product on the right hand side is over all pairs $(i,j)\in \dbN\times \dbN$
such that $i\alpha+j\beta\in \Phi$, in some fixed order, and
$C_{ij\alpha\beta}$ are integers independent of $R$ (but depending on the order).
Note that the product is finite by Proposition~\ref{prenil_basic}(a).

For $1\leq i\leq d$ and $r\in R^{\times}$ set 
$$\sgal_{i}(r)=x_{\alpha_i}(r)x_{-\alpha_i}(-r^{-1})x_{\alpha_i}(r),
\quad \sgal_{i}=\sgal_{i}(1)\quad\mbox{ and }\quad h_{i}(r)=\sgal_{i}(r)\sgal_{i}^{-1}.$$
The remaining relations are

(R3)
$\sgal_{i} x_{\alpha}(r)\sgal_{i}^{-1}=
x_{s_{i}\alpha}(\pm r)$

(R4) $h_{i}(r)x_{\alpha}(u)h_{i}(r)^{-1}=
x_{\alpha}(ur^{\la \alpha_i^{\vee},\alpha\ra})$ for $r\in R^{\times}$

(R5) $\sgal_{i}h_j(r)\sgal_i^{-1}=h_j(r)h_i(r^{-a_{ji}})$ for $r\in R^{\times}$

(R6) $h_{i}(ru)=h_{i}(r)h_{i}(u)$ for $r,u\in R^{\times}$

(R7) $[h_{i}(r),h_{j}(u)]=1$ for $r,u\in R^{\times}$

\begin{Remark} (a) Even though it is not obvious from the above relations, one can show that the Kac-Moody group $\dbG_A(R)$ can be defined directly in terms
of the root system $\Phi$, without explicit reference to $A$. In view of this, we will occasionally write $\dbG_{\Phi}(R)$ instead of $\dbG_A(R)$.

(b) The subgroup $\widetilde W=\la \sgal_i: 1\leq i\leq d\ra$ of $\dbG_A(F)$ is usually not isomorphic to $W$;
however, there is always an epimorphism $\widetilde W\to W$ which sends $\sgal_i$ to $s_i$ for each $i$, whose kernel is a finite group of exponent $\leq 2$.
\end{Remark}

Groups $\dbG_A(R)$ given by the above presentation first appeared in Tits' paper~\cite{Ti}; however, they are not called Kac-Moody groups in \cite{Ti}. Instead Tits defines a {\it Kac-Moody functor} (corresponding to a fixed GCM $A$), which we denote by $\dbG^{Tits}_A$, 
to be a functor from commutative rings to groups satisfying certain axioms and shows
that, when restricted to fields, such a functor is unique (up to natural equivalence), and for a field $F$, the group 
$\dbG^{Tits}_A(F)$ is isomorphic to the group $\dbG_A(F)$ given by the above presentation. If $R$ is a ring which is not a 
field, it is not known whether the group $\dbG^{Tits}_A(R)$ is uniquely determined up to isomorphism by the functor axioms or 
whether the group $\dbG_A(R)$ satisfies these axioms. What is known (and already established in \cite{Ti}) is that there is a 
natural homomorphism from $\dbG_A(R)$ to $\dbG^{Tits}_A(R)$ mapping root subgroups onto root subgroups. Another possible 
definition of Kac-Moody groups over rings, which uses highest weight modules, is discussed in \cite{Al2} and \cite{CW} and 
generalizes the analogous definition in the case of fields~\cite{CG}. Yet another candidate is the subgroup of the 
Mathieu-Rousseau complete Kac-Moody group $\dbG_A^{ma}(R)$ generated by real root subgroups (see \cite[\S 3]{Ro}). It can be 
shown that there are natural homomorphisms from the groups $\dbG_A(R)$ to $\dbG_A^{ma}(R)$ and to representation-theoretic 
Kac-Moody groups. 

Since property $(T)$ is preserved by homomorphic images, in view of the above remarks, the statement
of Theorem~\ref{thm:main} remains true if the groups $\dbG_A(R)$ are replaced by groups generated by the (real) root subgroups in any of the Kac-Moody groups mentioned in the previous paragraph (with the same restrictions on $A$ and $R$).
\vskip .2cm

Before proceeding, we define several Steinberg-type groups which project onto $\dbG_A(R)$.
Let $St^{(2)}_{A}(R)$ be the group generated by the same set of symbols $\{x_{\alpha}(r): \alpha\in\Phi, r\in R\}$
 but only subject to relations (R1) and (R2)
and by $St^{(3)}_{A}(R)$ the group with the same generating set and relations (R1), (R2) and (R3), so that we have natural epimorphisms $$St^{(2)}_{A}(R)\to St^{(3)}_{A}(R)\to \dbG_A(R).$$ 

The group $St^{(2)}_{A}(R)$ is called the Steinberg group (corresponding
to the pair $(A,R)$) in \cite{Ti}. The group $St^{(3)}_{A}(R)$ does not seem to have a specific name in the literature; we point
it out since it is the largest quotient of  $St^{(2)}_{A}(R)$ for which we will be able to prove property $(T)$ in the setting of
our main theorem (thus, relations (R4)-(R7) will not be important for us). Finally, \cite{MR} and \cite{Al} use 
a different notion of Steinberg group -- in their terminology Steinberg group is certain quotient of $St^{(3)}_{A}(R)$ which projects onto $\dbG_A(R)$. We refer the reader to \cite{Al} for the precise definition and detailed discussion about the relationship between different Steinberg-type groups.

\subsection{Some examples and facts about Kac-Moody groups}

For every root $\alpha\in \Phi$ and every subset $S$ of $R$ we will set $X_{\alpha}(S)=\{x_{\alpha}(s): s\in S\}$ (considered as a subset of $\dbG_A(R)$). If $S$ is a subgroup of $(R,+)$, then $X_{\alpha}(S)$ is a subgroup isomorphic to $S$. We will write $X_{\alpha}=X_{\alpha}(R)$
whenever $R$ is clear from the context.

The groups
$\{X_{\alpha}\}$ are called the root subgroups of $\dbG_A(R)$. Define $\dbG_A^+(R)=\la X_{\alpha}: \alpha\in \Phi^+\ra$
to be the subgroup of $\dbG_A(R)$ generated by all positive root subgroups. 

We proceed with two basic examples of Kac-Moody groups.

\begin{Example}
\label{exKM1}
Let $A$ be a GCM of spherical type (that is, $A$ is a Cartan matrix) and
$\Phi=\Phi(A)$. 
\end{Example}

Given a commutative ring $R$, let $\dbE_{\Phi}(R)$ denote
the elementary subgroup of the simply-connected Chevalley group of type $\Phi$ over $R$.
Then there exists a natural epimorphism $\dbG_A(R)=\dbG_{\Phi}(R)\to \dbE_{\Phi}(R)$, which is an isomorphism whenever $R$ is a field.

\begin{Example}
\label{exKM2}
Let $d\geq 2$ be an integer and define the $d\times d$ matrix $A=(a_{ij})$
by $a_{ij}=-1$ if $i-j\equiv \pm 1\mod d$, $a_{ij}=2$ if $i=j$ and
$a_{ij}=0$ otherwise. Then $\Phi=\Phi(A)$ is an affine root system of type
$\widetilde A_d$.
\end{Example}

Given a commutative ring $R$, there exists an epimorphism
$\pi:\dbG_A(R)\to EL_d(R[t,t^{-1}])$ given by
\begin{align*}
\pi(x_{\alpha_i}(r))=E_{i,i+1}(r) \mbox{ for } 1\leq i\leq d-1, &\quad
\pi(x_{\alpha_d}(r))=E_{d,1}(rt),\\
\pi(x_{-\alpha_i}(r))=E_{i+1,i}(r) \mbox{ for } 1\leq i\leq d-1, &\quad
\pi(x_{-\alpha_d}(r))=E_{1,d}(rt^{-1}).
\end{align*}
As in Example~\ref{exKM1}, $\pi$ is an isomorphism if $R$ is a field.
The group $\pi(\dbG^+_A(R))$ coincides with the subgroup of $EL_d(R[t])$
consisting of matrices which have upper-unitriangular image under the projection $EL_d(R[t])\to EL_d(R)$ which sends $t$ to $0$.

\vskip .2cm

We now return to the general case.

\begin{Theorem}
\label{thm:Al1} 
Let $A$ be a $2$-spherical GCM of size $d$ and $M=\max\{-a_{ij}: i\neq j\}$ (thus, $M\leq 3$). Let $R$ be a commutative ring
which does not have proper ideals of index $\leq M$. Then $\dbG_A^+(R)$ is generated by $\{X_{\alpha_i}\}_{i=1}^d$,
the root subgroups corresponding to simple roots. 
\end{Theorem}

\begin{proof} If $d=2$, Theorem~\ref{thm:Al1} is part of the assertion of \cite[Lemma~11.1]{Al}. In the general case let $H$ denote the subgroup of $\dbG_A^+(R)$ generated by $\{X_{\alpha_i}\}_{i=1}^d$. It is easy to show that $X_{\alpha}\subseteq H$
for every $\alpha\in\Phi^+$ by induction on the height of $\alpha$ using the result in the case $d=2$ and \cite[Lemma~6.2]{CER}.
\end{proof}

The proof of \cite[Lemma~11.1]{Al} mentioned above uses the precise commutation relations between positive root subgroups in the rank $2$ case. Below we list 
those relations since most of them will
be explicitly used later in the paper. In all four cases below we denote the simple roots of $\Phi$ by $\alpha$ and $\beta$ with $\alpha$ being the long root. We shall only list non-trivial commutator relations, that is, relations between root subgroups which do not commute.

{\it Case 1: $A=\begin{pmatrix} 2 & 0\\ 0& 2\end{pmatrix}$, $\Phi=A_1\times A_1$, $\Phi^+=\{\alpha,\beta\}$}. In this case
$X_{\alpha}$ and $X_{\beta}$ commute.

{\it Case 2: $A=\begin{pmatrix} 2 & -1\\ -1& 2\end{pmatrix}$, $\Phi=A_2$, $\Phi^+=\{\alpha,\beta,\alpha+\beta\}$}. In this case
\begin{align}
[x_{\alpha}(r),x_{\beta}(s)]=x_{\alpha+\beta}(rs)
\end{align}(for a suitable choice of Chevalley basis).

{\it Case 3: $A=\begin{pmatrix} 2 & -2\\ -1& 2\end{pmatrix}$, $\Phi=B_2$, $\Phi^+=\{\alpha,\beta,\alpha+\beta,\alpha+2\beta\}$}. In this case we have the following relations:
\begin{align}
\label{comm rel b2 one}[x_\alpha(r), x_{\beta}(s)]&= x_{\alpha+\beta}(rs)
x_{\alpha+2\beta}(rs^2)\\
\label{comm rel b2 two}[x_{\alpha+\beta}(r),x_\beta(s)] &= x_{\alpha+2\beta}(2rs).
\end{align}

{\it Case 4: $A=\begin{pmatrix} 2 & -3\\ -1& 2\end{pmatrix}$, $\Phi=G_2$, $\Phi^+=\{\alpha,\beta,\alpha+\beta,\alpha+2\beta,\alpha+3\beta,2\alpha+3\beta\}$}. In this case we have the following relations:
\begin{align}
\label{comm rel g2 one}[x_\alpha(r),x_\beta(s)] &=x_{\alpha+\beta}(rs)x_{\alpha+2\beta}(rs^2)x_{\alpha+3\beta}(rs^3)x_{2\alpha+3\beta}(r^2s^3),\\
\label{comm rel g2 two}[x_{\alpha+\beta}(r),x_{\beta}(s)] &=x_{\alpha+2\beta}(2rs)x_{\alpha+3\beta}(3rs^2)x_{2\alpha+3\beta}(3r^2s),\\
\label{comm rel g2 three}[x_{\alpha+2\beta}(r),x_{\beta}(s)] &=x_{\alpha+3\beta}(3rs),\\
\label{comm rel g2 four}[x_{\alpha+2\beta}(r),x_{\alpha+\beta}(s)] &=x_{2\alpha+3\beta}(3rs),\\
\label{comm rel g2 five}[x_{\alpha+3\beta}(r),x_{\alpha}(s)] &=x_{2\alpha+3\beta}(-rs).
\end{align}

\section{Property $(T)$}

We start by recalling the definition of property $(T)$ as well as the definition of relative property $(T)$
in the sense of \cite{Co2}. In the definitions below we allow $G$ to be an arbitrary topological group; however, we will deal
primarily with discrete groups, with Theorem~\ref{KMreals} being the only exception. 
For a general introduction to property $(T)$ we refer the reader to \cite{BHV}.

\begin{Definition}\rm
Let $G$ be a group and $S$ a subset of $G$.
\begin{itemize}
\item[(a)] Let $V$ be a unitary representation of $G$ and $\eps>0$. A vector $v\in V$ is called
$(S,\eps)$-invariant if $\|s v-v\|< \eps\| v\|$ for all $s\in S$.
\item[(b)] The Kazhdan constant $\kappa(G,S)$ is the largest $\eps\geq 0$ such that
if $V$ is any unitary representation of $G$ which contains an $(S,\eps)$-invariant vector, 
then $V$ contains a nonzero $G$-invariant vector.
\item[(c)] $S$ is called a {\it Kazhdan subset} of $G$ if $\kappa(G,S)>0$.
\item[(d)] $G$ has {\it property $(T)$} if it has a compact Kazhdan subset.
\end{itemize}
\end{Definition}

\begin{Definition}\rm

Let $G$ be a group and $B$ a subset of $G$.  The pair $(G,B)$ is said to have  {\it relative property $(T)$}
if there exist a compact subset $S$ of $G$  and a function $f:\R_{>0}\to \R_{>0}$ such that if $V$ is any
unitary representation of $G$ and $v\in V$ satisfies  $\|s v-v\|\le f(\eps) \|v\| $ for every $s\in S$, then
$\|b v-v\|\le \eps \|v\| $ for every $b\in B$.
\end{Definition}

In the case when $B$ is a normal subgroup, the above definition is equivalent to the following one
(which is the original definition of relative property $(T)$ -- see, e.g., \cite{BHV}):

\begin{Definition}\rm  
	\label{normal subgroup rel t}
Let $B$ be a normal subgroup of a group $G$.
\begin{itemize}
\item[(a)] A subset $S$ of $G$ is called a {\it relative Kazhdan subset} of the pair $(G,B)$
if there exists $\eps>0$ such that whenever a unitary representation $V$ of $G$ contains an $(S,\eps)$-invariant
vector, it must also contain a nonzero $B$-invariant vector. Any pair $(S,\eps)$ with this property is called a {\it Kazhdan pair} of $(G,B)$.

\item[(b)] The pair $(G,B)$ has relative property $(T)$	if it has a compact relative Kazhdan subset.
\end{itemize}
\end{Definition}

The only result about relative property $(T)$ which we will explicitly use in the proof of Theorem~\ref{thm:main} is the following
straightforward lemma.

\begin{Lemma}
\label{relativeT_trivial} Let $G$ be a group and $K$ a subgroup of $G$ with property $(T)$.
Then for any subset $H$ of $K$, the pair $(G,H)$ has relative property $(T)$.
\end{Lemma}

\begin{proof} Since $K$ has $(T)$, it is clear from the second definition of relative property $(T)$ that
the pair $(K,K)$ has relative $(T)$. And it is clear from the first definition that if a pair $(A,B)$
has relative $(T)$, then for any overgroup $A'\supseteq A$ and subset $B'\subseteq B$, the pair
$(A',B')$ has relative $(T)$. In particular, $(G,H)$ has relative $(T)$.
\end{proof} 

 In order to prove our main theorem we will use the ``almost orthogonality'' criterion for property $(T)$ based on the notion of orthogonality constant between subgroups of the same group. This method was originally introduced by Dymara and Januskieiwcz in \cite{DJ} and developed further in 
\cite{EJ,EJK,Ka2,Op2}.

\begin{Definition}\rm Let $H$ and $K$ be subgroups of the same group, and let $G=\la H,K\ra$ be the group
generated by them. 
\begin{itemize}
\item[(i)] Given a unitary representation $V$ of $G$, let $V^H$ and $V^K$ denote the subspaces
of $H$-invariant (resp. $K$-invariant) vectors in $V$. The orthogonality constant $orth(H,K;V)$
is defined by $$orth(H,K;V)=\sup\{|\la v,w\ra|: v\in V^H, w\in V^K, \|v\|=\|w\|=1\}.$$
\item[(ii)] The orthogonality constant $orth(H,K)$ is the supremum of the quantities $orth(H,K;V)$
where $V$ ranges over all unitary representations of $G$ without invariant vectors.
\end{itemize}
\end{Definition} 

We will use the following form of the almost orthogonality criterion. 

\begin{Theorem}\rm(\cite[Theorem~1.2]{EJ})
\label{propT_criterion}
Let $G$ be a discrete group and $H_1,\ldots, H_d$ subgroups of $G$. Suppose that $orth(H_i,H_j)<\frac{1}{d-1}$
for any $i\neq j$ and the pair $(G,H_i)$ has relative property $(T)$ for each $i$. Then
$G$ has property $(T)$.
\end{Theorem}

The following lemma collects two important cases where the orthogonality constant is small.

\begin{Lemma}
\label{lem:orthconst}
The following hold:
\begin{itemize}
\item[(a)] (see \cite[Lemma~3.4]{EJ}) If $H$ and $K$ are subgroups of the same group, and one of them normalizes the other (e.g., if they commute), then $\orth(H,K)=0$.
\item[(b)] (special case of \cite[Corollary~4.7]{EJ}) Let $R$ be a countable associative ring and $m(R)$ the smallest index of a proper left ideal of $R$. Let $G=Heis(R)$ be the Heisenberg group
over $R$, that is, the group of $3\times 3$ upper-unitriangular matrices over $R$, let $H=E_{12}(R)$ and $K=E_{23}(R)$. Then 
$\orth(H,K)\leq \frac{1}{\sqrt{m(R)}}$.
\end{itemize}
\end{Lemma}

Here are two basic examples where property $(T)$ follows immediately from Theorem~\ref{propT_criterion} and 
Lemma~\ref{lem:orthconst}:

\begin{Example} 
\label{ex1}
Let $A$ be a $2$-spherical $d\times d$ GCM with simply-laced Dynkin diagram and $R$ any finite commutative ring which has no proper ideals of index at most $(d-1)^2$. Then the positive unipotent subgroup $\dbG_A^+(R)$ has property $(T)$.
\end{Example}

Here we let $\{H_1,\ldots, H_d\}$ be the simple root subgroups of $G=\dbG_A^+(R)$. The groups $H_i$ are
finite since $R$ is finite, so the assumption that  $(G,H_i)$ has relative property $(T)$ holds trivially.
Since $Dyn(A)$ is simply-laced, for any $i\neq j$ either $H_i$ and $H_j$ commute (if $a_{ij}=0$)
or there is an isomorphism $\la H_i, H_j\ra\to Heis(R)$ which sends $H_i$ to $E_{12}(R)$ and $H_j$ to $E_{23}(R)$
(if $a_{ij}=-1$). Hence $\orth(H_i,H_j)<\frac{1}{d-1}$ by Lemma~\ref{lem:orthconst}.

\begin{Example} 
\label{ex2}
Now let $d\geq 3$ be any integer and $R$ any finitely generated associative ring with no proper ideals of index at most $(d-1)^2$. Let $G=EL_d(R)$, the subgroup of $GL_d(R)$ generated by elementary matrices. Then
$G$ has property $(T)$.
\end{Example}
As proved in \cite{EJ}, the group $EL_d(R)$ actually has property $(T)$ without any restrictions on indices
of ideals in $R$, but this result requires a much more general version of the almost orthogonality method.
To deduce the result stated in Example~\ref{ex2} directly from Theorem~\ref{propT_criterion} and 
Lemma~\ref{lem:orthconst} we set $H_i=E_{i,i+1}(R)$ for $1\leq i\leq d-1$ and $H_d=E_{d,1}(R)$.
Then $\orth(H_i,H_j)<\frac{1}{d-1}$ as in Example~3.
The fact that the pairs $(G,H_i)$ have relative property $(T)$ follows from Kassabov's theorem \cite[Theorem~1.2]{Ka1}.

Note that the argument in Example~\ref{ex2} remains valid if we replace $EL_d(R)$ by the Steinberg group $St_{d}(R)$. Thus it also proves that the Kac-Moody group $\dbG_{A_{d-1}}(R)$ has property $(T)$ since $\dbG_{A_{d-1}}(R)$ is a quotient of $St_{d}(R)$ as observed in \S~2. As we will see in \S~5, our general argument for property $(T)$ for Kac-Moody groups over rings will in some sense generalize Example~\ref{ex2} even though the collection of
subgroups $\{H_i\}$ to which Theorem~\ref{propT_criterion} will be applied in the case of groups of type
$A_{d-1}$ is different from the one in Example~\ref{ex2}.

\section{Orthogonality constants in Chevalley groups of rank $2$}

{\bf Notation:} Let $G$ be a group generated by subgroups $X$ and $Y$ and let $H$ be another subgroup of $G$. 
Let $m(H,X,Y)\in\dbN\cup\{\infty\}$ be the minimal dimension of an irreducible representation $V$ of $G$
such that $H$ acts non-trivially on $V$ and the subspaces $V^X$ and $V^Y$ are both nonzero.

The following result is a variation of \cite[Theorem~10.8]{EJK} and is proved by essentially the same argument.

\begin{Theorem}
\label{orth_step}
Let $G$ be a countable
group generated by subgroups $X$ and $Y$. Let $H$ be a subgroup of $Z(G)$, denote by $X'$ and $Y'$ the images of $X$ and $Y$ in $G/H$,
respectively, and let $m=m(H,X,Y)$. Then
$\orth(X,Y)\leq \sqrt{\orth(X', Y')+\frac{1}{m}}$.
\end{Theorem}
\begin{proof}
Let $\eps=\orth(X',Y')$.
By \cite[Claim~10.7]{EJK} is suffices to prove that $\orth(V^X,V^Y)\leq \sqrt{\eps+\frac{1}{m}}$
for every non-trivial {\it irreducible} representation $V$ of $G$.

Let us fix such a representation $V$. We can assume that $V^X$ and $V^Y$ are both nonzero (otherwise the result is trivial). If $H$ acts trivially on $V$, then $V$ is a representation of $G/H$, so there is nothing to prove. Thus,
we can assume that $H$ acts non-trivially, and therefore $\dim(V)\geq m$ by assumption.

Let $HS(V)$ denote the set of all
Hilbert-Schmidt  operators on $V$, that is, linear operators
$A:V\to V$ such that
$\sum_i \|A(e_i)\|^2$ is finite where $\{e_i\}$ is an
orthonormal basis  of $V$. Then $HS(V)$ is a Hilbert space with the
inner product given by
$$
\langle A,B\rangle=\sum_i \langle A(e_i),B(e_i)\rangle.
$$
It is easy to see that the set $HS(V)$ and the inner product do not depend on the choice of
$\{e_i\}$.

The representation of $G$ on $V$ yields the corresponding representation of $G$
on $HS(V)$, where an element $g\in G$ acts on $A\in HS(V)$  by
$(gA)(v)=gA(g^{-1}v).$  Since $H\subseteq Z(G)$ and $V$ is an irreducible representation of $G$, 
by Schur's lemma $H$ acts by scalars on $V$ and hence trivially on $HS(V)$. Thus, $HS(V)$
becomes a representation of $G/H$.

For a nonzero vector $v\in V$ let $P_v\in HS(V)$ denote the (orthogonal) projection onto $\dbC v$.
For a subspace $W$ of $HS(V)$ let  $\pi_W:HS(V)\to HS(V)$ denote the projection onto $W$.
The following results are proved in \cite{EJK}:
\begin{itemize}
\item[(a)] \cite[Lemma~10.1]{EJK}) $\la P_u,P_w\ra=|\la u,w\ra|^2$ for any unit vectors $u,w\in V$;
\item[(b)] \cite[Lemma~10.3]{EJK}) $\|\pi_{HS(V)^G}(P_v)|^2=\frac{1}{\dim(V)}$ for any unit vector $v\in V$.
\end{itemize}
The assertion of Theorem~\ref{orth_step} follows easily from these two results.
Indeed, take any unit vectors $u\in V^X$ and $w\in V^Y$. Then by (a)
$$|\la u,w\ra|^2=\la P_u,P_w\ra=\la \pi_{HS(V)^G}(P_u),\pi_{HS(V)^G}(P_w)\ra+\la \pi_{(HS(V)^G)^{\perp}}(P_u),\pi_{(HS(V)^G)^{\perp}}(P_w)\ra.$$
By (b) we have $$|\la \pi_{HS(V)^G}(P_u),\pi_{HS(V)^G}(P_w)\ra|\leq \|\pi_{HS(V)^G}(P_u)\|\cdot \|\pi_{HS(V)^G}(P_w)\|=\frac{1}{\dim(V)}\leq
\frac{1}{m}.$$ On the other hand, ${(HS(V)^G)}^{\perp}$ is a representation of $G/H$ without invariant vectors, so
$|\la \pi_{(HS(V)^G)^{\perp}}(P_u),\pi_{{(HS(V)^G)}^{\perp}}(P_w)\ra|\leq 
\eps\|\pi_{(HS(V)^G)^{\perp}}(P_u)\|\cdot\| \pi_{(HS(V)^G)^{\perp}}(P_w)\|\leq \eps$, which finishes the proof.
\end{proof}

\begin{Lemma} 
\label{lem:comrel}
Let $R$ be a countable commutative ring and $m(R)$ the smallest index of a proper ideal of $R$. Let $\Phi$ be an irreducible finite root
system of rank $2$, and let $\{\alpha,\beta\}$ be a base of $\Phi$, with $\alpha$ a long root. 

Define the group $G$ and its subgroups $X,Y$ and $H$ by one of the following:
\begin{itemize}
\item[(a)] $\Phi=A_2$, $G=\dbG_{\Phi}^+(R)$, $X=X_{\alpha}, Y=X_{\beta}$ and $H=X_{\alpha+\beta}$

\item[(b)] $\Phi=B_2$, $G=\dbG_{\Phi}^+(R)$, $X=X_{\alpha}, Y=X_{\beta}$ and
$H=X_{\alpha+2\beta}$. 

\item[(c)] $\Phi=G_2$, $G=\dbG_{\Phi}^+(R)$,
$X=X_{\alpha}, Y=X_{\beta}$ and $H=X_{\alpha+3\beta} X_{2\alpha+3\beta}$
\end{itemize}
In case (b) assume that $2$ is invertible in $R$, and in case (c)
assume that $3$ is invertible in $R$. Then (in each case) 
$m(H,X,Y)\geq m(R)$
\end{Lemma}
\begin{Remark} Case (a) of Lemma~\ref{lem:comrel} has already been established in \cite{EJ}; however, we have chosen to reproduce the proof as the arguments in other cases are similar, with additional technicalities involved.
\end{Remark}
\begin{proof} In each case we start with an arbitrary irreducible representation $V$ of $G$ with $V^X\neq \{0\}$ and $V^Y\neq \{0\}$ on which $H$ acts non-trivially. Our goal is to show that $\dim(V)\geq m(R)$.
Since $H\subseteq Z(G)$ by assumption, there exists a non-trivial character $\lam:H\to S^1$ such that each $h\in H$ acts on $V$ as the scalar $\lam(h)$.

(a) For brevity we set $\lam(r)=\lam(x_{\alpha+\beta}(r))$ for $r\in R$.
By assumption there exists nonzero $v\in V^X$. The commutator relation $[x_{\alpha}(r),x_{\beta}(s)]=x_{\alpha+\beta}(rs)$ (with $r,s\in R$ arbitrary) implies that 
$$x_{\alpha}(r) x_{\beta}(s) v=x_{\beta}(s)x_{\alpha}(r)[x_{\alpha}(r),x_{\beta}(s)] v=
\lam(rs) x_{\beta}(s)x_{\alpha}(r)v=\lam(rs)x_{\beta}(s) v.$$ 
Thus, for every $s\in R$, the vector $x_{\beta}(s) v$ is an eigenvector for $X$ with character $\lam_s: X\to S^1$ given by
$\lam_s(x_{\alpha}(r))=\lam(rs)$. Let $I=\{s\in R: \lam(rs)=0 \mbox{ for all }r\in R\}$. Then it is clear that $I$ is an ideal of $R$;
moreover, $I\neq R$ since $\lam$ is non-trivial, and therefore, $|R/I|\geq m(R)$. On the other hand, $\lam_s=\lam_t$ if and only if
$s\equiv t\mod I$, and therefore the number of distinct characters of the form $\lam_s$ is at least $m(R)$. Since eigenvectors corresponding
to distinct characters must be linearly independent, we conclude that $\dim(V)\geq m(R)$ as desired.

(b) This time we set $\lam(r)=\lam(x_{\alpha+2\beta}(r))$ for $r\in R$. Again choose any nonzero $v\in V^X$. Since  $2$ is invertible in $R$, the set $\{s\in R: \lam(2rs)=0 \mbox{ for all }r\in R\}$ is a proper ideal of $R$, and   arguing as in (a), this time using the relations
$[x_{\alpha+\beta}(r),x_\beta(s)] = x_{\alpha+2\beta}(2rs)$, we conclude that $\dim(V)\geq m(R)$.

(c) First assume that $X_{2\alpha+3\beta}$ acts non-trivially on $V$. Then the result
follows directly from (a) since there is an isomorphism between 
$\la X_{\alpha},X_{\alpha+3\beta},X_{2\alpha+3\beta}\ra$ and $St_{A_2^+}(R)$ which sends
$X_{2\alpha+3\beta}$ to $X_{\alpha+\beta}$.

Assume now that $X_{2\alpha+3\beta}$ acts trivially on $V$. Then by assumption $X_{\alpha+3\beta}$
must act non-trivially, and moreover $V$ is a representation of $G'=G/X_{2\alpha+3\beta}$.
The following relation holds in $G'$:
$$[x_{\alpha+2\beta}(r),x_\beta(s)]=x_{\alpha+3\beta}(3rs).$$
Since $3$ is invertible in $R$ and $X_{\alpha+3\beta}$ is a central subgroup of $G'$ which acts non-trivially, 
arguing as in (a), we conclude that $\dim(V)\geq m(R)$.
\end{proof}

Combining Theorem~\ref{orth_step} and Lemma~\ref{lem:comrel}, we can now estimate orthogonality constants between simple root subgroups in Chevalley groups of rank $2$.

Given a positive real number $m$, define the sequence $s_0(m),s_1(m),\ldots$
by $s_0(m)=0$ and $s_{i}(m)=\sqrt{s_{i-1}(m)+\frac{1}{m}}$ for all $i\geq 1$.

\begin{Corollary}
\label{cor:orth}
Let $R$ be a countable commutative ring and $m=m(R)$ the smallest index of a proper ideal of $R$. Let $\Phi$ be a finite root
system of rank $2$, and let $\{\alpha,\beta\}$ be a base of $\Phi$, with $\alpha$ a long root. Let $G=\dbG^+_{\Phi}(R)$, $X=X_{\alpha}(R)$ and
$Y=X_{\beta}(R)$. The following hold:
\begin{itemize}
\item[(a)] If $\Phi=A_1\times A_1$, then $\orth(X,Y)=0$
\item[(b)] If $\Phi=A_2$, then $\orth(X,Y)\leq s_1(m)=\frac{1}{\sqrt{m}}$
\item[(c)] If $\Phi=B_2$ and $2$ is invertible in $R$, then
 $\orth(X,Y)\leq s_2(m)<\sqrt[4]{\frac{3}{m}}$
\item[(d)] If $\Phi=G_2$ and $2$ and $3$ are invertible in $R$, then
 $\orth(X,Y)\leq s_4(m)< \sqrt[16]{\frac{188}{m}}$
\end{itemize}
\end{Corollary}
\begin{proof} In case (a) $X$ and $Y$ commute, so we are done by Lemma~\ref{lem:orthconst}(a).
Each of the subsequent cases follows from the previous one using Theorem~\ref{orth_step}, Lemma~\ref{lem:comrel}
and the following isomorphisms which send simple root groups to simple root groups:
$\dbG_{A_2}^+(R)/X_{\alpha+\beta}\cong \dbG_{A_1\times A_1}^+(R)$,
 $\dbG_{B_2}^+(R)/X_{\alpha+2\beta}\cong \dbG_{A_2}^+(R)$,
 $\dbG_{G_2}^+(R)/\la X_{\alpha+3\beta},X_{2\alpha+3\beta}\ra\cong \dbG_{B_2}^+(R)$
 (for the reduction of $G_2$ to $B_2$ we need to apply Theorem~\ref{orth_step} twice).
\end{proof}

\section{Proof of the main theorem and some variations}

In this section we will establish Theorem~\ref{thm:main} and discuss some of its variations. 
Theorem~\ref{thm:main} will be obtained as an easy
consequence of Corollary~\ref{cor:orth} and the following theorem: 

\begin{Theorem}
\label{thm:prenilpotentgeneration}
Let $A$ be a $2$-spherical $d\times d$ GCM whose indecomposable components have size
at least two (equivalently, the Dynkin diagram of $A$ has no isolated vertices). 
Let $R$ be a commutative ring and $m(R)$ the minimal index of a proper ideal of $R$,
and assume that $m(R)>\max\{-a_{ij}: i\neq j\}$. Let $G=\dbG_A(R)$ and $\Phi=\Phi(A)$ the
associated real root system. Then there exists a subset $\Sigma$
of $\Phi$ with $|\Sigma|<2d$ such that
\begin{itemize}
\item[(a)] $\gamma+\delta\neq 0$ for any $\gamma,\delta\in\Sigma$;
\item[(b)] for any $\gamma,\delta\in \Sigma$ either $X_{\gamma}$ and $X_{\delta}$ commute
or there exist $\alpha_i,\alpha_j\in \Pi$ and $w\in W$ such that $w\gamma,w\delta\in\dbZ\alpha_i+\dbZ\alpha_j$.
\item[(c)] the set $\cup_{\gamma\in \Sigma}X_{\gamma}$ generates $G=\dbG_A(R)$.
\end{itemize}
\end{Theorem}

We will first prove Theorem~\ref{thm:main} assuming Theorem~\ref{thm:prenilpotentgeneration} and then prove Theorem~\ref{thm:prenilpotentgeneration}.

\begin{proof}[Proof of Theorem~\ref{thm:main}]
We will prove that $G$ has property $(T)$ by applying Theorem~\ref{propT_criterion} to the collection of subgroups $\{H_{\gamma}\}_{\gamma\in \Sigma}$, where $\Sigma$ satisfies the conclusion of Theorem~\ref{thm:prenilpotentgeneration}. 

First we show that the pair $(G,X_{\alpha})$ has relative property $(T)$ for every $\alpha\in\Sigma$.
By relations (R3) in the definition of $\dbG_A(R)$, replacing $X_{\alpha}$ by a conjugate, we can assume that 
$\alpha=\alpha_i$ is a simple root. Since $A$ is indecomposable, there
exists $j\neq i$ such that $a_{ij}\neq 0$. Let $K=\la X_{\pm\alpha_i}, X_{\pm\alpha_j}\ra\subseteq G$.
Let $\Phi_{i,j}=\Phi(A_{\{i,j\}})$ (where $A_{\{i,j\}}=\begin{pmatrix}2& a_{ij}\\ a_{ji}& 2\end{pmatrix}$). 
Since $A$ is $2$-spherical and $a_{ij}\neq 0$, $\Phi_{i,j}$ is a root system of type $A_2,B_2$ or $G_2$, and
it is clear from the defining relations that $K$ is a quotient of the Steinberg group $St_{\Phi_{i,j}}(R)$. The group $St_{\Phi_{i,j}}(R)$
has property $(T)$ by \cite{EJK}, whence $(G,X_{\alpha})$ has relative property $(T)$ by 
Lemma~\ref{relativeT_trivial}.

It remains to check the required upper bounds on orthogonality constants. Take any
$\gamma,\delta\in\Sigma$. By Theorem~\ref{thm:prenilpotentgeneration} either
$X_{\gamma}$ and $X_{\delta}$ commute, in which case $\orth(X_{\gamma},X_{\delta})=0$,
or there exist $\alpha_i,\alpha_j\in\Pi$ and $w\in W$ such that $w\gamma,w\delta\in\dbZ\alpha_i+\dbZ\alpha_j$.
In the latter case, after conjugation in $G$, we can assume that $w=1$. 
Since $\gamma+\delta\neq 0$, an easy case-by-case verification using commutator relations in Chevalley groups shows that there exists a finite root system $\Psi$ of rank $2$ and an epimorphism $\dbG^+_{\Psi}(R)\to \la X_{\alpha_i},X_{\alpha_j}\ra$ which sends simple root subgroups of $\dbG^+_{\Psi}(R)$ to $X_{\alpha_i}$ and $X_{\alpha_j}$; moreover, $\Psi\neq B_2$ if $a_{ij}a_{ji}\leq 1$ and $\Psi\neq G_2$
if $a_{ij}a_{ji}\leq 2$. Applying Corollary~\ref{cor:orth} (and recalling the assumption on $R$ in the Theorem~\ref{thm:main}), we deduce that $\orth(X_{\alpha_i},X_{\alpha_j})<\frac{1}{2d-2}\leq \frac{1}{|\Sigma|-1}$.
\end{proof}

\begin{proof}[Proof of Theorem~\ref{thm:prenilpotentgeneration}]

Recall that $\Pi=\{\alpha_1,\ldots,\alpha_d\}$ denotes the set of simple roots and $Dyn(A)$ is the Dynkin diagram of $A$.
Let $\Pi_1$ be a maximal subset of $\Pi$ with the property that no two roots in $\Pi_1$ are connected to each other in $Dyn(A)$.
Let $\Pi_2=\Pi\setminus\Pi_1$, $k=|\Pi_1|$ and $l=|\Pi_2|=d-k$. Without loss of generality we can assume that $\Pi_1=\{\alpha_1,\ldots,\alpha_k\}$.
For brevity set $\beta_i=\alpha_{k+i}$ for $1\leq i\leq l$, so that $\Pi_2=\{\beta_1,\ldots,\beta_l\}$.

Let $w_0=s_{\alpha_1}\ldots s_{\alpha_k}$, set $\gamma_i=w_0(-\beta_i)$ for $1\leq i\leq l$, and let $\Sigma=\Pi\sqcup\{\gamma_i\}_{i=1}^l$. 
Clearly $|\Sigma|<2d$ and $\Sigma$ has property (a). We will now prove that $\Sigma$ also satisfies (b) and (c).
\vskip .2cm

(b) Let $\gamma,\delta\in\Sigma$. If $\gamma,\delta\in \Pi$ or if $\gamma,\delta\in\{\gamma_i\}$,
there is nothing to prove. Thus, after possibly swapping $\gamma$ and $\delta$, we can assume that
$\delta=\alpha_i$ for some $1\leq i\leq d$ and $\gamma=\gamma_j=w_0(-\beta_j)$ for some $1\leq j\leq l$

{\it Case 1:} $\alpha_i\neq\beta_j$ (that is, $i\neq j+k$). In this case any element of $\dbN\gamma+\dbN\delta$ clearly has both positive and negative coefficients (when expressed as a linear combination of simple roots), hence cannot be a root. Therefore, the pair $\{\gamma,\delta\}$
is prenilpotent by Proposition~\ref{prenil_basic}(a), and $X_{\gamma}$ and $X_{\delta}$ commute by relations (R2).

{\it Case 2:} $\alpha_i=\beta_j$. We have $\gamma_j=-\beta_j-\sum_{t=1}^k n_t \alpha_t$ where each $n_t\geq 0$; moreover $n_t>0$ if and only if $\alpha_t$ is connected to $\beta_j$ in $Dyn(A)$. Thus, $\la\gamma,\delta^{\vee}\ra=\la \gamma_j,\beta_j^{\vee}\ra\geq -2+m$ where $m$ is the number of roots in $\Pi_1$ connected to $\beta_j$ (by assumption $m\geq 1$). If $m=1$, then
$\gamma,\delta\in \dbZ\alpha_t\oplus \dbZ\beta_j$ where $\alpha_t$ is the unique root in $\Pi_1$ which is connected to $\beta_j$.
Assume now that $m\geq 2$. Then $\la\gamma,\delta^{\vee}\ra\geq 0$, whence by Proposition~\ref{prenil_basic}(c), $\{\gamma,\delta\}$ is prenilpotent
and the intersection $(\dbN\gamma+\dbN\delta)\cap \Phi$ is either empty or equals $\{\gamma+\delta\}$. In the former case we are done as in case 1,
and the latter case is actually impossible. Indeed $\gamma+\delta=-\sum_{t=1}^k n_t \alpha_t$ with at least two coefficients $n_t$ positive, whence
$\gamma+\delta$ is not a root since the roots $\{\alpha_t\}_{t=1}^k$ are pairwise disconnected.
\vskip .2cm

(c) Let $H$ be the subgroup generated by $\cup_{\gamma\in \Sigma}X_{\gamma}$. We need prove that $H=G$, for which it is sufficient to check that $H$ contains $X_{-\omega}$ for every simple
root $\omega$. We will argue by induction on $d$.

{\it Base case $d=2$}. In this case 
$\Phi=A_2,B_2$ or $G_2$ and, following our earlier convention, we denote the elements of
$\Pi$ by $\alpha$ and $\beta$, with $\alpha$ being a long root.

Note that if $\{\gamma,\delta\}$ is any base of $\Phi$, then $\{\gamma,\delta\}$ (considered as an unordered pair) 
is conjugate to $\Pi$; hence by Theorem~\ref{thm:Al1} and our assumption on $R$, for any root 
$\eps\in \dbN\gamma+\dbN\delta$ the root subgroup $X_{\eps}$
lies in $\la X_{\gamma},X_{\delta}\ra$.  
\vskip .1cm

{\it Subcase 1:} $\Phi=A_2$, $\alpha_1=\alpha$, $\beta_1=\beta$. In this case $\gamma_1=s_{\alpha}(-\beta)=
-(\alpha+\beta)$. Since $\{\gamma_1,\beta\}$ is a base of $\Phi$
and $-\alpha=\gamma_1+\beta$, we have $X_{-\alpha}\subseteq H$.

Once we know that $H$ contains $X_{\alpha}$ and $X_{-\alpha}$,
we conclude that $\widetilde s_{\alpha}\in H$, whence
$X_{-\beta}=
X_{s_{\alpha}(\gamma_1)}={\widetilde s_{\alpha}}^{-1}X_{\gamma_1}\widetilde s_{\alpha}\subseteq H$,
and we are done.

{\it Subcase 2:} $\Phi=B_2$, $\alpha_1=\alpha$, $\beta_1=\beta$. In this case $\gamma_1= -(\alpha+\beta)$. 

Since $\{\gamma_1,\alpha\}$ is a base and
$-(\alpha+2\beta)=2\gamma_1+\alpha$, the subgroup $H$ contains $X_{-(\alpha+2\beta)}$. Since $\{-(\alpha+2\beta),\beta\}$
is a base and $-\alpha=-(\alpha+2\beta)+2\beta$,
we conclude that $H$ contains $\widetilde s_{\alpha}$, and we can finish
the proof as in subcase 1.

{\it Subcase 3:} $\Phi=B_2$, $\alpha_1=\beta$, $\beta_1=\alpha$. In this case $\gamma_1=s_{\beta}(-\alpha)= -(\alpha+2\beta)$. 

Since $\{-(\alpha+2\beta),\beta\}$ is a base and $-(\alpha+\beta)=-(\alpha+2\beta)+\beta$,
we have $X_{-(\alpha+\beta)}\subseteq H$, hence we are done by subcase 2.

{\it Subcase 4:} $\Phi=G_2$, $\alpha_1=\alpha$, $\beta_1=\beta$. In this case $\gamma_1=-(\alpha+\beta)$, and the argument is analogous to subcase 2 with $-(\alpha+2\beta)$ replaced by $-(\alpha+3\beta)$.

{\it Subcase 5:} $\Phi=G_2$, $\alpha_1=\beta$, $\beta_1=\alpha$. In this case $\gamma_1=-(\alpha+3\beta)$, and the argument is analogous to subcase 3, again with $-(\alpha+2\beta)$ replaced by $-(\alpha+3\beta)$.

\vskip .2cm
We proceed with the induction step. Let $d>2$, and assume (c) has
been established for all matrices of size less than $d$. If
$Dyn(A)$ is disconnected, the induction step is trivial, so we can assume that $Dyn(A)$ is connected.

{\it Case 1:} Each simple root in $\Pi_1$ is connected to a root
of $\Pi_2$ different from $\beta_1$. In this case the Dynkin subdiagram with vertex set $\Pi\setminus\{\beta_1\}$ has no isolated vertices and $\Pi_1$ is a maximal subset of pairwise disconnected vertices
in $\Pi\setminus\{\beta_1\}$. Hence, by induction hypothesis $H$ contains
$X_{-\gamma}$ for every $\gamma\in\Pi$ except possibly $\gamma=\beta_1$. In particular, $X_{-\alpha_i}\subseteq H$
for all $1\leq i\leq k$. Since $-\beta_1=s_{\alpha_1}\ldots s_{\alpha_k}(\gamma_1)$ and $X_{\gamma_1}\subseteq H$,
arguing as in subcase 1 of the base step,
we conclude that $X_{-\beta_1}\subseteq H$, and we are done.

{\it Case 2:} There exists a simple root in $\Pi_1$ which is only
connected to $\beta_1$. Without loss of generality assume that
$\alpha_1$ is the root with this property. Since we assume that $Dyn(A)$ is connected,
$\beta_1$ must be connected to a simple root other than $\alpha_1$.
Hence the Dynkin subdiagram with vertex set $\Pi\setminus\{\alpha_1\}$
has no isolated vertices, and it is clear that
$\Pi_1\setminus\{\alpha_1\}$ is a maximal subset of pairwise disconnected vertices in $\Pi\setminus\{\alpha_1\}$.

Since $\alpha_1$ is only connected to $\beta_1$, we have
$s_{\alpha_1}\ldots s_{\alpha_k}(-\beta_j)=s_{\alpha_2}\ldots s_{\alpha_k}(-\beta_j)$ for all $j\neq 1$. Hence,
by induction hypothesis $H$ contains
$X_{-\gamma}$ for every $\gamma\in\Pi$ except possibly $\gamma=\beta_1$ and $\gamma=\alpha_1$.

In particular, $H$ contains $X_{\pm \alpha_i}$ for all $2\leq i\leq k$. Since $X_{\gamma_1}\subseteq H$ and $s_{\alpha_2}\ldots s_{\alpha_k}(\gamma_1)=s_{\alpha_1}(-\beta_1)$, it follows that
that $H$ contains $X_{s_{\alpha_1}(-\beta_1)}$. Applying the result
in the base case to the Dynkin subdiagram with vertex set
$\{\alpha_1,\beta_1\}$, we conclude that $H$ contains $X_{-\alpha_1}$
and $X_{-\beta_1}$. The proof is complete.
\end{proof}

\subsection{Some variations}

As we already saw in Example~\ref{ex1}, if $R$ is a finite commutative ring
and $A$ is a 2-spherical GCM with simply-laced Dynkin diagram, then the positive subgroup $\dbG_A^+(R)$ of the Kac-Moody 
group $\dbG_A(R)$ has property $(T)$ whenever $R$ has no proper ideals of small index. The proof of Theorem~\ref{thm:main}
shows that the result remains true without the assumption that $A$ is simply-laced.
As already mentioned in the introduction, if $R$ is infinite, the group $\dbG_A^+(R)$ has infinite abelianization and thus 
cannot have property $(T)$; however, Theorem~\ref{thm:pseudoparabolic} below shows that one can still construct many subgroups with property $(T)$ which lie between $\dbG_A^+(R)$ and $\dbG_A(R)$.

Let $A$ be a GCM of size $d$ and $I$ a subset of $\{1,2,\ldots,d\}$. Given a commutative ring $R$, define $P_{A,I}(R)$ to the
subgroup of $\dbG_A(R)$ generated by all simple root subgroups $X_{\alpha_i}(R), 1\leq i\leq d$ as well as the negative root subgroups $X_{-\alpha_i}(R), i\in I$ (thus, $\dbG_A(R)=P_{A,I}(R)$ for $I=\{1,2,\ldots,d\}$ and $\dbG_A^+(R)=P_{A,\emptyset}(R)$
provided the hypotheses of Theorem~\ref{thm:Al1} hold). We will call
the groups $P_{A,I}(R)$ the {\it pseudo-parabolic} subgroups of $\dbG_A(R)$ (the parabolic subgroups, which we will not consider, 
are defined in the same way except that they must also contain the standard torus).

\begin{Theorem}
\label{thm:pseudoparabolic}
Let $A=(a_{ij})$ be an indecomposable 2-spherical GCM of size $d\geq 2$, assume that $Dyn(A)$ has no triple edges, and define $n(A)$ as in Theorem~\ref{thm:main}. Let $I$ be a subset of 
$\{1,2,\ldots,d\}$ such that for every $1\leq i\leq d$ there exists $j\in I$ such that $j\neq i$ and $a_{ij}\neq 0$. If
$R$ is any finitely generated commutative ring which does not have proper ideals of index less than $n(A)$,
then the pseudo-parabolic subgroup $P_{A,I}(R)$ has property $(T)$.
\end{Theorem}
\begin{Remark} $\empty$

1. Theorem~\ref{thm:pseudoparabolic} remains valid even if $Dyn(A)$ has a triple edge. This generalization, which requires new results on relative property $(T)$, will be proved in the appendix to this paper by Zezhou Zhang.

2. If $R$ surjects onto $\dbZ$, it is easy to show that the condition on $I$ in the statement of Theorem~\ref{thm:pseudoparabolic}
is necessary for $P_{A,I}(R)$ to have property $(T)$. We do not know if there are any infinite rings for which 
$P_{A,I}(R)$ has property $(T)$ without the above condition on $I$.

\end{Remark} 
\begin{proof} The proof of Theorem~\ref{thm:pseudoparabolic} is essentially the same as that of Theorem~\ref{thm:main}, requiring just small modifications. Let $I_1$ be a maximal subset of $I$ with the property that any two simple roots $\alpha_i,\alpha_j$ with $i,j\in I_1$, are not connected. Let $I_2=I\setminus I_1$, $I_3=\{1,\ldots,d\}\setminus I$ and $\Pi_k=\{\alpha_i: i\in I_k\}$ for $k=1,2,3$.

Let $w=\prod_{i\in I_1}s_{\alpha_i}$, let $\Lambda=w(-\Pi_2)=\{-w(\alpha_j): j\in I_2\}$ and $\Sigma=\Pi\cup \Lambda$.
We claim that $\Sigma$ satisfies conditions (a) and (b) of Theorem~\ref{thm:prenilpotentgeneration} and
the set $\{X_{\gamma}: \gamma\in \Sigma\}$ generates $P_{A,I}(R)$. Condition (a) is obvious. Next we check (b) --
it is clear if $\gamma,\delta\in \Pi$ or $\gamma,\delta\in \Lambda$ and holds by the proof of Theorem~\ref{thm:prenilpotentgeneration} if $\gamma\in \Pi_1\cup \Pi_2$ and $\delta\in \Lambda$ (or vice versa). If  $\gamma\in \Pi_3$ and 
$\delta\in \Lambda$ (or vice versa), then no simple root appears in the expansion of both $\gamma$ and $\delta$;
since $\gamma$ and $\delta$ have opposite signs, $X_{\gamma}$ and $X_{\delta}$ commute 
by the argument from Case 1 of the proof of Theorem~\ref{thm:prenilpotentgeneration}(b).
Finally, applying the proof of Theorem~\ref{thm:prenilpotentgeneration} to the Dynkin subdiagram $Dyn_I(A)$, we conclude that
the subgroup $\la X_{\gamma}: \gamma\in \Pi_1\cup \Pi_2\cup \Lambda\ra$ is equal to $\la X_{\gamma}: \gamma\in \pm(\Pi_1\cup \Pi_2)\ra$,
whence $\{X_{\gamma}: \gamma\in \Sigma\}$ must generate $P_{A,I}(R)$.

To finish the proof it remains to show that the pair $(P_{A,I}(R), X_{\gamma})$ has relative property $(T)$ for every $\gamma\in \Sigma$
(once this is done, we simply repeat the argument in the proof of Theorem~\ref{thm:main}), and this is where our hypothesis on $I$ comes into play. First of all, by assumption the subdiagram $Dyn_I(A)$ has no isolated vertices, hence we can apply 
Theorem~\ref{thm:main} to the submatrix $A_I$ to conclude that the group $\la X_{\gamma}: \gamma\in \Pi_1\cup \Pi_2\cup \Lambda\ra$ has property $(T)$. In particular, this implies that $(P_{A,I}(R), X_{\gamma})$ has relative $(T)$ for $\gamma\in \Pi_1\cup \Pi_2\cup \Lambda$. 

If $\gamma\in \Pi_3$, we choose $\delta\in \Pi_1\cup \Pi_2$ which is connected to $\gamma$ (such $\delta$ exists
by assumption on $I$). Let $\Psi=(\dbZ\gamma\oplus \dbZ\delta)\cap \Phi$. By assumption $\Psi$ is a root subsystem of type $A_2$ or $B_2$. Consider the subgroup $H=\la X_{\gamma}, X_{\delta},X_{-\delta}\ra$. It is easy to see that $H=\la X_{\delta},X_{-\delta}\ra N$ where $N=\la X_{\lam} : \lam\in\Psi^+\setminus\{\delta\}\ra$ and $N$ is normal in $H$.
Since $H\subseteq P_{A,I}(R)$ and $X_{\gamma}\subseteq N$, to finish the proof it suffices to show that $(H,N)$ has relative $(T)$. We consider three cases.

{\it Case 1:} $\Psi$ is of type $A_2$. Then there exists an epimorphism $St_2(R)\ltimes R^2\to H$ which sends $St_2(R)$ to $\la X_{\delta},X_{-\delta}\ra$
and $R^2$ to $N=X_{\gamma}X_{\gamma+\delta}$. Since the pair $(St_2(R)\ltimes R^2, R^2)$ has relative property $(T)$ by \cite[Appendix~A]{EJK}, it follows that $(H,X_{\gamma}X_{\gamma+\delta})$ has relative property $(T)$.

{\it Case 2:} $\Psi$ is of type $B_2$ and $\gamma$ is a short root. In this case 
$(H,N)$ has relative property $(T)$ by \cite[Corollary~7.11]{EJK}.

{\it Case 3:} $\Psi$ is of type $B_2$ and $\gamma$ is a long root. In this case 
$N\cong (S^2(R^2),+)$, where $S^2$ denotes the second symmetric power,
and the action of $\la X_{\delta},X_{-\delta}\ra$ on $N$ factors through the corresponding action of $EL_2(R)$ on $S^2(R^2)$.
The pair $(EL_2(R)\ltimes S^2(R^2),S^2(R^2))$ has relative property $(T)$ by \cite[Theorem~3.3]{Ne}, and since $N$ is abelian, 
\cite[Corollary~2]{CT} implies that $(H,N)$ has relative property $(T)$ as well.
\end{proof} 

If $R$ is a ring which is not finitely generated, Kac-Moody groups over $R$ cannot possibly have property $(T)$ as discrete groups;
however, they may still have property $(T)$ when endowed with suitable topology. In particular, the following theorem
was proved by Hartnick and K\"ohl in \cite{HK}:\footnote{The theorem proved in \cite{HK} is slightly more general as it deals with almost split Kac-Moody groups while we only consider split Kac-Moody groups}

\begin{Theorem}
\label{KMreals}
Let $A$ be an indecomposable 2-spherical $d\times d$ GCM, with $d\geq 2$, and $F$ a local field. Then the Kac-Moody group $\dbG_A(F)$ endowed with the Kac-Peterson topology has property $(T)$.
\end{Theorem}
We finish the paper by providing a new proof of this theorem. We refer the reader to \cite{HKM,HK} for the definition of the Kac-Peterson topology.

\begin{proof} The following proof which substantially simplifies our original argument was suggested by Pierre-Emmanuel Caprace.
First, it is clear from the definition of property $(T)$ that if $G$ is a topological group, $\Gamma$ a dense subgroup of $G$ and $\Gamma$ has property $(T)$
when considered as a discrete group, then $G$ has property $(T)$. 

For a subring $R$ of $F$, denote by $\dbG_A(R,F)$ the subgroup of $\dbG_A(F)$ generated
by $\cup_{\alpha\in\Phi(A)}X_{\alpha}(R)$. It is clear from the defining relations that $\dbG_A(R,F)$ is a homomorphic image of $\dbG_A(R)$. 
On the other hand, it is clear from the definition of the Kac-Peterson topology
that if $R$ is a dense subring of $F$, then $\dbG_A(R,F)$ is dense in $\dbG_A(F)$. 

Thus, to prove that $\dbG_A(F)$ with the Kac-Peterson topology has property $(T)$, it suffices to find a dense subring $R$ of $F$ such that $\dbG_A(R)$ has property $(T)$ as a discrete group. In view of Theorem~\ref{thm:main}, it is enough to show that for any $n\in\dbN$
there exists a dense finitely generated subring $R$ of $F$ with no proper ideals of index at most $n$. If $F=\dbR$, we set $R=\dbZ[\frac{1}{n!}]$. If $F=\dbC$,
we set $R=\dbZ[i,\frac{1}{n!}]$. If $F$ is a finite extension of $\dbQ_p$, choose $\alpha\in F$ such that $F=\dbQ_p[\alpha]$
and set $R=\dbZ[\alpha,\frac{1}{p\cdot n!}]$. Finally, if $F=k((t))$, with $k$ a finite field, we set $R=k[t,\frac{1}{t\cdot f_n}]$ where $f_n$ is the product of all irreducible polynomials in $k[t]$ of degree at most $\log_{|k|}(n)$. Clearly, in each case $R$ has the required property. (assuming $n\geq 2$).
\end{proof}

\appendix

\section{}
\centerline{by \sc Zezhou Zhang}
\renewcommand{\thesubsection}{\Alph{subsection}}
\vskip .5cm

In this appendix we develop the necessary tools to generalize, towards the end, Theorem~\ref{thm:pseudoparabolic} of the main text to allow triple edges. As we do not see generality as a top priority, we value conciseness over presenting results in their strongest form.

Most of this appendix will be devoted to proving a new result regarding relative property~($T$).
The reader may note the resemblance between our proof and classical approaches from \cite{Shalom, EJK, Burger}.

\begin{Definition}\rm
	 For any commutative ring $R$, let $V_2(R)$ be the rank 2 free $R$-module equipped with the standard action of $\mathfrak{sl}_2(R)$ (and $\SL_2(R)$). For $m \geq 0$, we define $V_{m+1}(R)$ as the symmetric power $Sym^m(V_2(R))$, which is free of rank $m+1$, and carries a canonical $\SL_2(R)$-module structure. 
\end{Definition}
 Note that in the case $R=\dbR$, $V_n(\rr)$ is the unique irreducible $\sltwo(\rr)$-module of dimension $n$. We now state the main result.

\begin{Theorem}\label{one}
	Let $R$ be a finitely generated unital commutative ring. Given the discrete topology, the pair  $(\EL_2(R) \ltimes V_n(R), V_n(R))$ has relative property ($T$) for $2 \leq n \leq 4$, where $\EL_2(R)$ is the subgroup of $\SL_2(R)$ generated by all shear matrices. 
\end{Theorem}
\begin{remarkstar}\mbox{}
	
	\begin{itemize}
		\item Cases $n=2$ and $n=3$ are proved by Shalom \cite{Shalom} and Neuhauser \cite{Ne}, respectively.
		\item 	The theorem remains true if we substitute $\SL_2(R)$ for $\EL_2(R)$.
		\item $V_n(R) \cong R^n $ as (left) $R$-modules.
		\item 
		It is well known that $\EL_2(\zz)=\SL_2(\zz)$.
	\end{itemize}
\end{remarkstar}

For the sake of subsequent arguments, we elucidate here how the shear matrices in $\SL_2(R)$ act on $V_n(R)$.  

Identify $V_n(R)$ with the space of column vectors $R^n$, and let $s \in R$, and denote the representation map as $\rho_n$. Then $\rho_n(E_{12}(s))$ is given by  
\[ \rho_n(E_{12}(s))_{ki}=\begin{cases}
{{n-k}\choose{i-k}}s^{i-k} \quad &\mbox{ if $k \leq i \leq n$}\\
\ 0  &\mbox{ if $i <k$}
\end{cases}
\]
while 
\[ \rho_n(E_{21}(s))_{ki}=\begin{cases}
{{k-1}\choose{k-i}}s^{k-i} \quad &\mbox{ if $ i \leq k$}\\
\ 0  &\mbox{ if $k < i \leq n $}
\end{cases}
\]
In other words, these two matrices have rows given by expansion of the polynomials $(1+s)^m$. An example for $n=4$ is available on the next page.

For notational convenience, we will denote the matrices $\rho_n(E_{12}(s))$ by $\boldsymbol{U_+^s}$ and $\rho_n(E_{21}(s))$ by $\boldsymbol{U_-^s}$ in the sequel, and may, by abuse of notation, use the $\rho_n(E)$ and $U$'s interchangeably.

By \cite[Proposition 4.3]{Co}, $(\SL_2(\rr) \ltimes V_n(\rr), V_n(\rr))$ is a pair with relative property ($T$). As it was observed in \cite[Section 4.4]{Co}, the following lemma is a clear consequence of 
$\SL_2(\zz) \ltimes V_n(\zz)$ being a lattice in $\SL_2(\rr) \ltimes V_n(\rr)$:
\begin{Lemma}\label{t for z}
$(\SL_2(\zz) \ltimes V_n(\zz), V_n(\zz))$ is a pair with relative property \mbox{(T)} for $n \geq 2$. 
\end{Lemma}

Proving Theorem \ref{one} thus amounts to  generalizing Lemma~\ref{t for z} to general $R$. Since relative property ($T$) is preserved by quotients, the finite generation assumption on $R$  allows us to reduce to the case where $R$ is a finitely generated polynomial ring $\zz[t_1, \ldots t_l]$, which is clearly also a single variable polynomial ring when we write it as $(\zz[t_1\ldots t_{l-1}])[t_l]$. So we assume from now on that $R$ is a  polynomial ring over $\dbZ$.
\vskip.3cm

 Like in \cite[Section 3.2]{Shalom}, we use $K$ to denote ``the additive group of formal power series" $\widehat{R}[[t^{-1}]]=\{\sum_{k=0}^{\infty} \chi_k t^{-k} \mid \chi_k \in \widehat{R}\}$, where $\widehat{R}= \mbox{Hom}(R, \cc^\times)$.  It is topologically isomorphic to $\widehat{R[t]}$, the (Pontryagin) dual of the polynomial ring $R[t]$. We also consider its enlarged version ``additive group of Laurent  series in $t^{-1}$'' $\widehat{R}((t^{-1}))=\{\sum_{k=n}^{\infty} \chi_k t^{-k} \mid \chi_k \in \widehat{R}, n\in \zz \}$, which we denote as $\widetilde{K}$. As $R$ is a polynomial ring over $\zz$, we see that $\widetilde{K}$ is a $t^l$-torsion free $R[t]$ module (for all $l \in \dbZ^+$) under the action defined by standard multiplication of power series, where $(r\chi)(s)=\chi(rs), \ r,s\in R, \chi \in \widehat{R}$. In other words, no nontrivial element in $\widetilde{K}$ is annihilated by $t^l$. It is clear that the free $R[t]$ module $V_n(R[t]) \cong R[t]^n$ has dual isomorphic to ${K}^n$, which embeds in $\widetilde{K}^n$. 
 
 Coming back to $\EL_2(R[t]) \ltimes V_n(R[t])$, note that $\EL_2(R[t])$ acts on $K^n$, and consequently on $\widetilde{K}^n$, by the usual contragradient action: writing $\widetilde{K}^n$ as row vectors $(x_1, \ldots x_n)$ where  $x_i \in \widetilde{K}$, the action of $E_{12}(s)$ (resp. $E_{21}(s)$) corresponds to right multiplication by the matrix $U_+^{-s}$ (resp. $U_-^{-s}$). 
 
 \vskip .1cm
 \noindent\textit{\textbf{Example}}. When $n=4$, 
 \begin{eqnarray*}
 	U_+^{s}=\begin{pmatrix}
 		1&3s&3s^2&s^3\\
 		0& 1& 2s& s^2\\
 		0& 0& 1 & s\\
 		0&0&0&1\\
 	\end{pmatrix} \quad \quad U_-^{s}=\begin{pmatrix}
 	1&0&0&0\\
 	s& 1& 0& 0\\
 	s^2& 2s& 1 & 0\\
 	s^3&3s^2&3s&1\\
 \end{pmatrix}
\end{eqnarray*}

Their action on $\widetilde{K}^4$ may be explicity written as:
\begin{eqnarray}
(a, b, c, d). U_+^s &= &(a, 3sa+b, 3s^2a+2sb+c, s^3a+s^2b+sc+d ) \label{Ktildaaction1}\\
(a, b, c, d). U_-^s &= &(a+sb+s^2c+s^3d, b+2sc+3s^2d, c+3sd, d ) \label{Ktildaaction2}
\end{eqnarray}

 The following lemma is the pivotal ingredient of this appendix:

\begin{Lemma}\label{TBA}
	Let $\mu$ be a mean (i.e. a finitely additive probability measure) defined on the Borel sets of $\widetilde{K}^4 \backslash \{0\}
	$. Then there exists
	a Borel set $S \subset \widetilde{K}^4 \backslash \{0\}$ and an element $g \in O_t = 
	\{U_\pm^{\pm t}, U_\pm^{ \pm  1}\}$ such that $\left| \mu(S.g) - \mu (S) \right| \geq \frac{1}{22}$.
	\end{Lemma}

\begin{remarkstar}
	Note that $O_t$ is the set $\{E_{12}(\pm 1), E_{21}(\pm 1),E_{12}(\pm t),  E_{21}(\pm t)\} \subset \EL_2(R[t])$.
\end{remarkstar}

\begin{proof}
	
	Let us assume the contrary, namely all Borel subsets of $\widetilde{K}^4 \backslash \{0\}$ are $(O_t, C)$-invariant, where $C=\frac{1}{22}$.
	
Define $v: \widetilde{K} \rightarrow \zz \ \cup \{-\infty\}$ on $\widetilde{K}$ using the highest power of $t$ in $\chi$:
	\[
	v(\chi)=\begin{cases}
	{-m} \quad &\mbox{ if } \chi=\sum\limits_{n=m}^{+\infty} \chi_n t^{-n}, \chi_m\neq 0 \\
	-\infty \quad &\mbox{ if } \chi=0.\\
	\end{cases}
	\]
Then $v$ enjoys the properties  $v(t\chi)=v(\chi)+1$ and $v(\chi+\psi) \leq max\{v (\chi) , v (\psi)\}$, where for $v (\chi) \neq v (\psi)$ even equality holds.  
	
	For any $\chi =\sum_{n=m}^{\infty}\chi_n t^{-n} \in \widetilde{K}$ (where $\chi_m \neq 0$), we define the ``leading term map" 
$\dbL(\chi)=\chi_m t^{-m}$. It might be worth noting that this map is $\zz$-linear.

		 Now define in $\widetilde{K}^4 \backslash \{0\}$ the sets $A_i=\{\chi=(x_1, \ldots ,x_4) | v(x_i)= \max\limits_{1\leq j \leq 4}v(x_j) \}$: these are the elements in $\widetilde{K}^4$ with maximal valuation attained at the $i$-th component. We also define $A_i^o=\{\chi\in A_i | v(x_i)> \max\limits_{1\leq j \leq 4, i\neq j}v(x_j) \}$, where maximal valuation is strictly attained at the $i$-th component. It is clear that $\widetilde{K}^4 \backslash \{0\}=\bigcup\limits_{i=1}^{4}A_i$. 
		 
Now we proceed with computations, which will be based on the equations $(\ref{Ktildaaction1}), (\ref{Ktildaaction2})$.
		  
The right multiplication by $U_+^{1}$ (resp. $U_-^{1}$) sends
		 		\begin{eqnarray}\label{1}
		 			A_2^o \sqcup 
		 			A_3^o \ \  \rightarrow \ \  A_{4}\backslash A_{4}^o \ ( \mbox{resp. }A_{1}\backslash A_{1}^o).  
		 			\end{eqnarray}
		 			
		 			Let us work out here the first case of (\ref{1}) as an example of the reasoning behind the calculations we are doing : Let $(x_1,x_2,x_3,x_4) \in A_2^o$, i.e. the maximal valuation among $x_1,x_2,x_3,x_4$ is attained \textbf{only} at $x_2$. So by setting $s=1$ in  (\ref{Ktildaaction1}), it follows that $(x_1,x_2,x_3,x_4).U^1_+=(x_1, 3x_1+x_2, 3x_1+2x_2+x_3, x_1+x_2+x_3+x_4 ) $ has maximal valuation attained at least at the second and fourth components; namely $U_+^1$ sends $A_2^o$ into $A_4 \backslash A_4^o$. For $(x_1,x_2,x_3,x_4) \in A_3^o$, i.e. the maximal valuation among $x_1,x_2,x_3,x_4$ is attained \textbf{only} at $x_3$, the same calculation shows that $(x_1,x_2,x_3,x_4).U^1_+$ attains the maximal valuation at the third and fourth components; namely $U_+^1$ sends $A_3^o$ into $A_4 \backslash A_4^o$.
		 			
		 			Arguing using methods similar to the previous paragraph, we see that the  right multiplication by $ U_+^{t}$, followed by $ U_-^{1} $ sends 
		 	\begin{eqnarray}\label{2}
		 		A_1 \rightarrow A_{4}^o \rightarrow  \{\chi \in A_{1}| v(x_1)=v(x_4) \}=:E.
		 			\end{eqnarray} 
		 			
		 			By symmetry, we get that the  right multiplication by $U_-^{t}\cdot U_+^{1} $ sends		\begin{eqnarray}\label{3}
		 				A_4 \rightarrow A_{1}^o \rightarrow \{\chi \in A_{4}| v(x_1)=v(x_4) \}=E.
		 			\end{eqnarray} 	 		
		 			
		 			Now by (\ref{2},\ref{3}) and $ {C}$-invariance of $\mu$, we derive that \begin{eqnarray}\label{4}
		 			\mu(A_1\backslash E) <2C , \mbox{ and } \ \mu(A_4\backslash E) <2C.
		 			\end{eqnarray}
		 			
		 			Similarly to (\ref{2}), the right multiplication by $U_+^{t} \cdot U_-^{t}$ or $U_-^{t} \cdot U_+^{t}$  sends 
		 			\begin{eqnarray}\label{5}
		 			A_1 \rightarrow A_{4}^o \rightarrow  A_1^o \mbox{ or } A_4 \rightarrow A_{1}^o \rightarrow  A_4^o,  \mbox{ respectively }.
		 			\end{eqnarray} 
		 			
		 			This implies the bounds 
		 			\begin{eqnarray}\label{6}
		 			\mu(A_1\backslash A_1^o) <2C , \ \mu(A_4\backslash A_4^o) <2C.
		 			\end{eqnarray}
		 			
		 			Combining (\ref{1}, \ref{6}) and $ {C}$-invariance of $\mu$, we yield \begin{eqnarray}\label{7}
		 			\mu(A_2^o \sqcup 
		 			A_3^o ) <C+\mu(A_1\backslash A_1^o)<3C.
		 			\end{eqnarray}
		 			
		 			Noting that $E \cap (A_1^o \cup A_4^o)=\emptyset$, (\ref{4}), together with  the $2C$ bounds on $\mu(A_1\backslash A_1^o)$ and $\mu(A_4\backslash A_4^o)$ tell us 
		 			\begin{eqnarray}\label{8}
		 			\mu(A_1) <4C \mbox{ and } \mu(A_4) <4C
		 			\end{eqnarray}
		 			
		 			Thus we are now left to estimate the size of $B=(A_2 \cap A_3) \backslash (A_1 \cup A_4)$. We start by considering $S:= \{\chi \in B \  | \  \dbL(t^3 x_1)+ \dbL(t^2 x_2)=0 \}$. It is clear that $U_+^t . (B\backslash S) \subseteq A_4^o$, so we have by virtue of (\ref{4}) and $E \cap (A_1^o \cup A_4^o)=\emptyset$ that	$\mu(B\backslash S) < C+ \mu(A_4 \backslash
		 			 E)<C+2C=3C $. As for $S$, using the fact that there is no $t$-torsion in $\widetilde{K}^4$, we note that the equations $\dbL(t^3 x_1)+ \dbL(t^2x_2)=0$ (which corresponds to the last column of $U_+^t$) and $3\dbL(t^2 x_1)+ 2\dbL(t x_2)=0$(which corresponds to the 3rd column of $U_+^t$) cannot happen simultaneously as we are working in $B$. Therefore $U_+^t . (S) \subseteq A_3^o \sqcup A_4$, so we have by (\ref{4},\ref{6},\ref{7}) that  $\mu(S) < C+3C+4C=8C$.
		 			
		 			Summing up the computations above, we get
		 		\begin{eqnarray*}
		 			1=\mu(\widetilde{K}^4 \backslash \{0\}) \leq& \mu(A_1)+\mu(A_4)+\mu(A_2^o \sqcup 
		 			A_3^o )+\mu(B \backslash S) +\mu(S) \\ <& \newline 4C+4C+3C+3C+8C=22C =1, \quad \quad \quad 
		 		\end{eqnarray*} which is clearly a contradiction.	 \end{proof}

Let us fix some notations: for $s \in R[t]$, $V_n(R)s$ is the $V_n(R)$-submodule of $V_n(R[t])$ where all $n$ coordinates of its elements are 
$R$-multiples of $s$. It is clear that the group pairs $(\EL_2(R) \ltimes V_n(R)t^i, V_n(R)t^i), i\in \mathbb{Z}^{\geq 0}$ are mutually isomorphic. So if $(Q,\eps)$ is a Kazhdan pair of $(\EL_2(R) \ltimes V_n(R), V_n(R))$ (see Definition~\ref{normal subgroup rel t}), then  the images $Q_i$ of $Q$ under the isomorphisms, along with $\eps$, give Kazhdan pairs for $(\EL_2(R) \ltimes V_n(R)t^i, V_n(R)t^i)$. We also denote by $W_i$ the free 
$R$-modules $V_n(R)t^i$, and set  $W_{-1}=\{\mbox{the zero element in } V_n(R[t])\}$. 

Since we have already established Lemma \ref{t for z}, the remaining parts of Theorem \ref{one} is a straightforward consequence of applying successively the following

\begin{Proposition}\label{two}
	 Let $R$ be a finitely generated polynomial ring over $\zz$ such that $(Q,\eps)$ is a Kazhdan pair for  $(\EL_2(R) \ltimes V_n(R), V_n(R))$, and $\{E_{12}(\pm 1), E_{21}(\pm 1)\} \subseteq Q$. Define  $Q_t=Q\cup \bigcup\limits_{i=1}^{n-2}Q_i \cup \{E_{12}(\pm t), E_{21}(\pm t)\} $. 
	 Then for $n=4$, there exists a $\delta \in (0, \eps)$ such that the pair $(\EL_2(R[t]) \ltimes V_4(R[t]), V_4(R[t]))$ has relative property ($T$), with Kazhdan pair $(Q_t, \delta)$.  Moreover, $(Q_t,\delta)$ is a Kazhdan pair for $(\langle Q_t \rangle \ltimes V_4(R[t]),V_4(R[t]))$, and therefore the pair  $(\langle Q_t \rangle \ltimes V_4(R[t]),V_4(R[t]))$ has relative property $(T)$.
\end{Proposition}

\begin{remarkstar}
	As our proof is similar to \cite[Propsition 4.3.1]{BHV}, we encourage the reader to refer this source for certain omitted details.
\end{remarkstar}
\begin{proof}
	Let $(\pi, \mathcal{H})$ be a unitary representation of $\EL_2(R[t]) \ltimes V_4(R[t])$ which has a $(Q_t,\delta)$-invariant vector $\xi \in \mathcal{H}$ and {\it assume by contradiction that $\mathcal{H}$ has no non-zero vector invariant under $V_4(R[t])\cong (R[t])^4$}. Let $\mathcal{H}_0$ be the subspace of all $V_4(R)$-invariant vectors in $\mathcal{H}$. The proof consists of three steps:

	\noindent$\bullet$ {\it Step 1:} Obtain a $(Q_t, \frac{1}{44})$-invariant unit vector $\boldsymbol{\eta_0}$ inside the subspace of $\bigcup\limits_{i=0}^2 V_4(R)t^i$-invariant vectors. 
\vskip .2cm	
	\noindent{\textbf{\textit{Claim}}}
	Let $-1\leq j \leq 1$ be an integer and $\delta_j \in (0,\eps)$ a real number. If $\mathcal{H}$ contains  a $(Q_t,\delta_j)$-invariant unit vector $v$ in $\bigcap\limits_{i=-1}^j\mathcal{H}^{W_i}$ (i.e. the  subspace of $\bigcup\limits_{i=-1}^j W_i$-invariant vectors), then it also contains, for some $\delta_{j+1}>0$, a $(Q_{t},\delta_{j+1})$-invariant unit vector $w$ in $\bigcap\limits_{i=-1}^{j+1}\mathcal{H}^{W_i}$. Moreover, $\delta_{j+1} \rightarrow 0$ as $\delta_j \rightarrow 0$.
	
\begin{proof}[Proof of Claim] Denote $\mathcal{H}_{[0,j]}:=\bigcap\limits_{i=0}^j\mathcal{H}^{W_i}$.
It admits an orthogonal decomposition $\mathcal{H}_{[0,j]}=\mathcal{H}_{[0.j+1]}\oplus \mathcal{H}_{[0.j+1]}^{c}$, yielding $v=v_0+v_1$.

Noting that all $V_4(R)t^i$ and $V_4(R[t])$ are abelian normal subgroups of $\EL_2(R) \ltimes V_n(R[t])$, we see that both $\mathcal{H}_{[0.j+1]}$ and $\mathcal{H}_{[0.j+1]}^{c}$ are $\EL_2(R) \ltimes V_n(R[t])$-subrepresentations in $\mathcal{H}$. Therefore for all $g \in Q_t$, 
\[ \norm{\pi(g)v_1-v_1}^2 \leq \norm{\pi(g)v-v}\leq \delta_j^2 \leq \eps^2 .\]

As $\mathcal{H}_{[0.j+1]}^{c}$ contains no non-zero $W_{j+1}$-invariant vector, and $(Q_{j+1}, \eps)$ is a Kazhdan subset for $(\EL_2(R) \ltimes V_n(R)t^{j+1}, V_n(R)t^{j+1})$, there exists $g_0 \in Q_{j+1}$ such that 
\[ \eps^2\norm{v_1}^2 \leq \norm{\pi(g_0)v_1-v_1}^2.\] Along with the previous inequality, this implies that $ \eps^2\norm{v_1}^2 \leq \delta_j^2$. Finally,  $\norm{v_0}^2 \geq 1-\delta_j^2/\eps^2$ is a consequence of $1= \norm{v_0}^2+\norm{v_1}^2$.

Thus for $w=v_0/\norm{v_0}$, we obtain for every $g\in Q_t$
\begin{eqnarray*}
\norm{\pi(g)w-w}^2 \leq& &\frac{1}{\norm{v_0}^2} (\norm{\pi(g)v-v}+\norm{\pi(g)v_1-v_1})^2\\ 
\leq& &\frac{1}{1-\delta_j^2/\eps^2}(\delta_j + 2\delta_j/\eps)^2(\boldsymbol{= :\delta_{j+1}^2}). 
\end{eqnarray*}
Noting that $0<\delta_j < \eps$, this equation proves the claim. 
\end{proof}

Thus by starting with $\mathcal{H}_{[0,-1]}=\mathcal{H}$ and choosing our initial $\delta$  such that $\delta/\eps$ is small enough, we can apply the Claim three consecutive times (covering cases $-1\leq j \leq 1$) to obtain our desired $(Q_t, \frac{1}{44})$-invariant unit vector $\boldsymbol{\eta_0}$.

Recall the notation $K=\widehat{R}[[t^{-1}]]$ and $\widetilde{K}=\widehat{R}((t^{-1}))$. Now let $E$ be the projection valued measure on $\widehat{V_4(R[t])}=\widehat{R[t]^4}(\cong K^4)$ given by the restriction of $\pi$ to $R[t]^4$. The fact that $\mathcal{H}$ has no nonzero 
$R[t]^4$-invariant vectors implies $E(\{0\})=0$. Denote by $X$ the subset of $\widehat{R[t]^4}$ of all elements $\chi$ such that $\chi|_{V_4(R+Rt+Rt^2)}=1$. This set enjoys the property that $E(X)$ is the orthogonal projection to $\bigcap\limits_{i=-1}^2\mathcal{H}^{W_i}$.

	\noindent$\bullet$ {\it Step 2:} 
	Let $\mu$ be the probability measure defined on the Borel sets of $K^4$ by $\mu(B)=\langle E(B)\boldsymbol{\eta_0} , \boldsymbol{\eta_0}\rangle$. Then $\mu(\{0\})=0$ and $\mu(X)=1$. As in \cite[Propsition 4.3.1, Steps 2,3]{BHV}, every Borel set $S$ of $K^4$ satisfies $|\mu(gS)-\mu(S)|<\frac{1}{22}$ for all $g\in O_t$ (defined as in Lemma \ref{TBA}). 
	
	It is clear that we may extend $\mu$ to a probability measure on $\widetilde{K}^4 \backslash\{0\}$.  Since the  identification of  $X$ with a subset of $K^4$  mandates that all four components of its elements must take the form $\sum\limits_{i=3}^{\infty} \chi_n t^{-i}, \ \chi_n \in \widehat{R}$, it follows that    $gX \subseteq K^4 \subset \widetilde{K}^4$   for all $g\in O_t$ for the group action on $\widetilde{K}^4$. Therefore the inequality of the previous paragraph applies to $\widetilde{K}^4 \backslash\{0\}$ as well: every Borel set $S$ of $\widetilde{K}^4$ satisfies $|\mu(gS)-\mu(S)|<\frac{1}{22}$ for all $g\in O_t$. This definitely contradicts Lemma \ref{TBA}, and the proof (for the main part of the proposition) is finished.
	
	Now the remaining part is actually an easy consequence: 
	Since in our proof above we used among all elements of $\EL_2(R[t])$ only the ones in $Q_t$ when we consider actions on $V_n$, the ``moreover'' part of the proposition clearly follows.
\end{proof}

We are now ready to generalize Theorem~\ref{thm:pseudoparabolic} :

\begin{Theorem}\label{original theorem}
	Let $A = (a_{ij})$ be an indecomposable $2$-spherical GCM of size $d\geq 2$ , and let $I$ be a subset of $\{1, 2, \ldots ,d\}$ such that for every
	$1 \leq i \leq d$ there exists $j \in  I$ such that $j \neq i$ and $a_{ij} \neq 0$. If $R$ is any finitely generated
	commutative ring which does not have proper ideals of index less than $(2d-2)^2$, then the pseudo-parabolic subgroup $P_{A,I} (R)$ has property ($T$).
\end{Theorem}

\begin{proof}
	Borrowing notations from Theorem~\ref{thm:pseudoparabolic}, let $\Psi=G_2$, with a fixed choice of simple roots $\Pi =\{\alpha, \beta \}$, where $\alpha$ is long and  $\beta$ is short. For $\delta, \gamma \in \Pi$, $\delta \neq \gamma $,  consider the subgroup $H = \langle X_\delta,X_{-\delta},X_\gamma \rangle$. It is easy to see that for $N = \langle X_\lambda \mid \lambda \in \Psi ^+ \backslash \{\delta \} \rangle $,  $H =\langle X_\delta, X_{-\delta}\rangle \ltimes  N$. Note that $H \subset P_{A,I}(R)$ and $X_\gamma \subset N$.
	
	The unproven part of the Theorem now hinges on		
		\begin{Claim}\label{final step}
			 $(H,X_\gamma)$ has relative property $(T)$.
		\end{Claim}
	\noindent{\it{Proof.}} Recall that with the set of simple roots $\Pi$ chosen as above,  $\Psi=\{ \pm \alpha, \pm \beta, \pm (\alpha+\beta),\pm (\alpha+2\beta),\pm (\alpha+3\beta), \pm (2\alpha+3\beta) \}$.
	
	\begin{center}
		$\ast$~$\ast$~$\ast$
	\end{center}
	{\it Case 1:} $\delta=\alpha$,  $\gamma=\beta$. 
	
		For $N_0:=\langle X_{\alpha+3\beta},X_{2\alpha+3\beta }\rangle \  \triangleleft \  H$, there is an embedding $\EL _2(R) \ltimes R^2 \cong \langle X_\alpha, X_{-\alpha} \rangle \ltimes N_0 \hookrightarrow H$. By the $n=2$ case of Theorem \ref{one}, $(\langle X_\alpha, X_{-\alpha} \rangle \ltimes N_0, N_0 )$ has relative $(T)$.
		
Let $\overline H=H/N_0$,  and for a subset $S$ of $H$ denote by $\overline S$ its projection in $H/N_0$, so in particular $\overline{N}=N/N_0$.

The first step of our proof would be the statement:
	
\begin{Claim}\label{last min change}
		$(\langle X_\alpha, X_{-\alpha} \rangle \ltimes \overline{N} , \overline{N})$ has relative property $(T)$. 
\end{Claim}

\begin{proof}[Proof of Claim \ref{last min change}] Analyzing the structure of $\langle X_\alpha, X_{-\alpha} \rangle \ltimes \overline{N}$,  one obtains the following: $\overline{N}=\langle \overline{X_\beta}, \overline{X_{\alpha+\beta}}, \overline{X_{\alpha+2\beta}}  \rangle$, $\overline{H}= \langle \overline{X_\beta}, \overline{X_{\alpha+\beta}}, \overline{X_{\alpha+2\beta}}, \overline{X_\alpha}, \overline{X_{-\alpha}}\rangle$, the subgroups $X_\alpha$, ${X}_\beta$, ${X}_{\alpha+\beta}$, ${X}_{\alpha+2\beta}$ and $X_{-\alpha}$ are isomorphic to their images under the quotient map, and $\langle X_\alpha, X_{-\alpha} \rangle \ltimes \overline{N} \cong \overline{H}$.

Using relations (\ref{comm rel g2 one}),(\ref{comm rel g2 two}) from the main text, we see that $\overline{H}$ satisfies  certain commutator relations (which we would like think of as ``$B_2$'' relations):
\begin{eqnarray*}
	[\overline{x_\alpha(r)},\overline{x_\beta(s)}]&=&\overline{x_{\alpha+\beta}(rs)}\ \overline{x_{\alpha+2\beta}(rs^2)},\\
	\mbox{} [\overline{x_{\alpha+\beta}(r)},\overline{x_\beta(s)}]&=&\overline{x_{\alpha+2\beta}(2rs)}.
\end{eqnarray*}

Comparing this with the commutator formulas (\ref{comm rel b2 one}),(\ref{comm rel b2 two}) from the main text, we see that $(\overline{H},\overline{N})$ is isomorphic to the pair $(H,N)$ from Case 2 of Theorem~\ref{thm:pseudoparabolic}. Therefore $(\overline{H},\overline{N})$ indeed has relative $(T)$, and Claim \ref{last min change} is proved.
\end{proof}
	
	We now invoke a lemma about extensions:
	
	\begin{Lemma}\rm (\it\cite[Lemma 5.3]{Fernos}\rm )
		Suppose that $0\rightarrow A_0 \rightarrow A \rightarrow A_1 \rightarrow 0$ is an exact sequence and that $\Gamma$ acts by automorphisms on $A$ and leaves $A_0$ invariant. If $(\Gamma \ltimes A_0, A_0 )$ and $(\Gamma \ltimes A_1, A_1)$ have relative property ($T$) then so does $(\Gamma \ltimes A, A)$.
	\end{Lemma} 
	
	It is now clear that $(H, N )$ has relative $(T)$. Since $X_\beta \subseteq N$, it follows that $(H,X_\beta)$ has relative property $(T)$, and Case 1 is proved.
	\begin{center}
		$\ast$~$\ast$~$\ast$
	\end{center}
\noindent	{\it Case 2:} $\delta$ is the short simple root $\beta$. 
	
	\noindent	We shall again show that $(H,N)$ has relative property ($T$). Then the same would hold for $(H, X_\gamma)$, as was displayed in the last paragraph of Case 1.

	Note that $X_{2\alpha+3\beta}$ is central in $H$, and that $N=\langle X_{\alpha}, X_{\alpha+\beta}, X_{\alpha+2\beta}, X_{\alpha+3\beta}, X_{2\alpha+3\beta} \rangle$. Let us again use the ``bar for quotient'' notation: denote $H/X_{2\alpha+3\beta}$ by $\overline{H}$, and denote the image of any root subgroup $X_\gamma$ under this quotient map as $\overline{X_\gamma}$. Under this notation, $\overline{N}=\langle \overline{X_\alpha}, \overline{X_{\alpha+\beta}}, \overline{X_{\alpha+2\beta}}, \overline{X_{\alpha+3\beta}} \rangle$.

	An inspection of the commutator formulas
	\begin{eqnarray*}
		[x_\alpha(r),x_\beta(s)]&=&x_{\alpha+\beta}(rs)x_{\alpha+2\beta}(rs^2)x_{\alpha+3\beta}(rs^3)x_{2\alpha+3\beta}(r^2s^3),\\
		\mbox{} [x_{\alpha+\beta}(r),x_\beta(s)]&=&x_{\alpha+2\beta}(2rs)x_{\alpha+3\beta}(3rs^2)x_{2\alpha+3\beta}(3r^2s),\\
		\mbox{} [x_{\alpha+2\beta}(r),x_\beta(s)]&=&x_{\alpha+3\beta}(3rs),\\
		\mbox{} 
		[x_{\alpha+3\beta}(r),x_\beta(s)]&=& 1
	\end{eqnarray*} and similar formulas involving $[X_{\alpha},X_{-\beta}],  [X_{\alpha+\beta},X_{-\beta}]$ etc. shows that
	
	\begin{itemize}
		\item $\overline{H} \cong \langle \overline{X_\beta}, \overline{X_{-\beta}} \rangle \ltimes \overline{N}$,
		\item $\overline{N}= \overline{X_{\alpha+3\beta}} \times \overline{X_{\alpha+2\beta}} \times \overline{X_{\alpha+\beta}} \times \overline{X_{\alpha}} \cong R^4$ (as an abelian group),
		\item the action of $\langle \overline{X_\beta}, \overline{X_{-\beta}} \rangle$ on $\overline{N}$ corresponds to the $\EL_2(R)$ action on $V_4(R)$.
	\end{itemize}  Thus Theorem \ref{one} implies that $(\overline{H},\overline{N})$ has relative ($T$). However, we cannot use this fact directly to prove relative $(T)$ for the pair $(H,N)$; we need to establish instead a slightly stronger result: there exists a finite subset $S$ of $H$ such that $( \overline{\langle S,N \rangle},\overline{N})$ has relative $(T)$.

The construction goes as follows: 

By applying Proposition A.5 $n$ times, we see that there exists a finite subset $Q_{[1,n]}$ of $\EL_2(\zz[x_1, \ldots x_n])$ such that the pair $( \la Q_{[1,n]} \ra \ltimes V_n(\zz[x_1, \ldots x_n]), V_n(\zz[x_1, \ldots x_n]))$ has relative property $(T)$ and moreover $Q_{[1,n]}$ is a relative Kazhdan subset for this pair. Assuming $R$ to be $n$-generated, let $proj: \EL_2(\zz[x_1, \ldots x_n]) \ltimes V_n(\zz[x_1, \ldots x_n]) \rightarrow \langle \overline{X_\beta}, \overline{X_{-\beta}}  \rangle \ltimes \overline{N}$ be the standard projection induced by a projection from   $\zz[x_1, \ldots x_n]$ to $R$, and set $\overline{S}=proj(Q_{[1,n]})$. As the image of the pair $( \la Q_{[1,n]} \ra \ltimes V_n(\zz[x_1, \ldots x_n]), V_n(\zz[x_1, \ldots x_n]))$ under $proj$, $(\langle \overline{S} \rangle \ltimes \overline{N}, \overline{N})$ clearly has relative ($T$), and the finite set $\overline{S}$ is its relative Kazhdan subset. Letting $S \subseteq H$ be a finite subset that surjects onto $\overline{S}$, we set $\boldsymbol{H'=\langle S,N \rangle}$.

It is now imperative to quote a theorem regarding extensions:  
	
	\begin{Theorem}[Theorem 10.5, EJK]\label{extensionlemma}
		Let $H'$ be a group, $N$ a normal subgroup of $H'$ and $Z \subseteq  Z(H') \cap  N$, where $Z(H')$ is the center of $H'$. Put
		$I = Z \cap [N, H']$. Let $S$ be a finite subset of $H'$.
		
		If the conditions
		\begin{enumerate}
			\item $S$ and $N$ generate $H'$, 
			\item $(H'/Z,N/Z)$ and $(H'/I,Z/I)$ both have relative ($T$)
		\end{enumerate}
		are both satisfied, then $(H',N)$ has relative property ($T$).
	\end{Theorem}

Once notations are set up properly according to \ref{extensionlemma} above, this theorem implies straightforwardly that $(H',N)$ has relative $(T)$: recall how $\overline{S}$, $S$, $\overline{H}$, $H$, and $H'$ are defined before Theorem \ref{extensionlemma}. Additionally, let $Z=X_{2\alpha+3\beta}$ -- this is a valid choice as $X_{2\alpha+3\beta}$ is central in $H$. It follows that  $I=X_{2\alpha+3\beta} \cap [N, \langle S, N \rangle ]=X_{2\alpha+3\beta}$. Noting that the pairs $(\langle \overline{S} \rangle \ltimes \overline{N}, \overline{N})= (H'/Z,N/Z)$ and $(\overline{H},\overline{N})$ both have $\overline{S}$ as a (finite) relative Kazhdan subset, and that $(H'/I,Z/I)=(H'/Z, Z/Z)$ trivially have relative  ($T$), Theorem \ref{extensionlemma} gives us the desired result.

 As $H'\subseteq H$, $(H,N)$ has relative property ($T$) as well. This finishes Case 2 and proves Claim \ref{final step}.
\end{proof}

\begin{remarkstar}
	The fact that $(\EL_2(R) \ltimes V_3(R), V_3(R))$ has relative ($T$) is used when $\Psi=B_2$, which is Case $3$ in the proof of Theorem~\ref{thm:pseudoparabolic}.
\end{remarkstar}

\section*{Acknowledgement}
The author would like to thank Marcus Neuhauser for his clarification about the computations in \cite{Ne}, and Mikhail Ershov for his careful checking of the results in this appendix.


\begin{thebibliography}{AAA}

\bibitem[Al1]{Al} Daniel Allcock,
\it Steinberg groups as amalgams,
\rm Algebra Number Theory 10 (2016), no. 8, 1791–-1843.

\bibitem[Al2]{Al2} Daniel Allcock,
\it Presentation of affine Kac-Moody groups over rings,
\rm Algebra and Number Theory 10 (2016), no. 3, 533--556. 


\bibitem[BHV]{BHV} Bachir Bekka, Pierre de la Harpe and Alain Valette,
\it Kazhdan's property $(T)$.
\rm New Math. Monogr. 11, Cambridge Univ. Press, Cambridge, 2008.

\bibitem[Bu]{Burger} Marc Burger, 
\it Kazhdan constants for $SL(3,\zz)$, 
\rm J. Reine Angew. Math. 431 (1991), 36--67.



\bibitem[CER]{CER} Lisa Carbone, Mikhail Ershov and Gordon Ritter, 
\it Abstract simplicity of complete Kac-Moody groups over finite fields. 
\rm J. Pure Appl. Algebra 212 (2008), no. 10, 2147–-2162.


\bibitem[CR2]{CR2} Pierre-Emmanuel Caprace and Bertrand R\'emy,
\it Groups with a root group datum. 
\rm Innov. Incidence Geom. 9 (2009), 5–-77.



\bibitem[CG]{CG} Lisa Carbone and Howard Garland,
\it Existence of lattices in Kac-Moody groups over finite fields.
\rm Commun. Contemp. Math.  5  (2003),  no. 5, 813--867

\bibitem[CW]{CW} Lisa Carbone and Frank Wagner,
\it Uniqueness of representation--theoretic hyperbolic Kac--Moody groups over $\Z$,
\rm preprint (2015), arXiv:1512.04623

\bibitem[Co]{Co} Yves de Cornulier,
\it Relative Kazhdan property,
\rm Ann. Sci. \'Ecole Norm. Sup. (4) 39 (2006), no. 2, 301--333.

\bibitem[Co2]{Co2} Yves de Cornulier, 
\textit{Kazhdan and Haagerup Properties in algebraic groups over local fields}. 
\rm J. Lie Theory 16 (2006), 67--82.

\bibitem[CT]{CT} Yves de Cornulier and Romain Tessera,
\it A characterization of relative Kazhdan property T for semidirect products with abelian groups. 
\rm Ergodic Theory Dynam. Systems 31 (2011), no. 3, 793–-805. 

\bibitem[DJ]{DJ} Jan Dymara and Tadeusz Januszkiewicz,
\it Cohomology of buildings and their automorphism groups,
\rm Invent. Math.  150  (2002),  no. 3, 579--627.


 \bibitem[EJ]{EJ} Mikhail Ershov and Andrei Jaikin-Zapirain,
\it Property $(T)$ for noncommutative universal lattices,
\rm Invent. Math. 179 (2010), no. 2, 303--347.

\bibitem[EJK]{EJK} Mikhail Ershov,  Andrei Jaikin-Zapirain and Martin Kassabov,
\it Property $(T)$ for groups graded by root systems,
\rm arXiv:1102.0031.v2, accepted by Mem. Amer. Math. Soc. 

\bibitem[Fe]{Fernos} Talia Fern\'{o}s,  
\textit{Relative property (T) and linear groups}. 
\rm Annales de l'institut Fourier, 56 (2006), no. 6, 1767--1804.

\bibitem[HK]{HK} Tobias Hartnick, Ralf K\"ohl, 
\it Two-spherical topological Kac-Moody groups are Kazhdan. 
\rm J. Group Theory 18 (2015), no. 4, 649–-654.

\bibitem[HKM]{HKM} Tobias Hartnick, Ralf K\"ohl and Andreas Mars,
\it On topological twin buildings and topological split Kac-Moody groups. 
\rm Innov. Incidence Geom. 13 (2013), 1–-71. 


\bibitem[Ka1]{Ka1} Martin Kassabov, 
\it Universal lattices and unbounded rank expanders,
\rm Invent. Math. 170 (2007), no. 2, 297--326.

\bibitem[Ka2]{Ka2} Martin Kassabov, 
\it{Subspace arrangements and property $(T)$},
\rm Groups Geom. Dyn. 5 (2011), no. 2, 445–-477.


\bibitem[Kac]{Kac} Victor Kac,
\emph{Infinite-dimensional Lie algebras. Third edition.}
\rm Cambridge University Press, Cambridge, 1990

\bibitem[KP1]{KP1} Victor Kac and Dale H. Peterson, 
\it Infinite-dimensional Lie algebras, theta functions and modular forms. 
\rm Adv. in Math. 53 (1984), no. 2, 125–-264.

\bibitem[KP2]{KP2} Victor Kac and Dale H. Peterson,
\it Defining relations of certain infinite-dimensional groups.
\rm The mathematical heritage of Elie Cartan (Lyon, 1984). Asterisque 1985, Numero Hors Serie, 165--208.
 

\bibitem[MR]{MR} Jun Morita and Ulf Rehmann,
\it Matsumoto-type theorem for Kac-Moody groups. 
\rm Tohoku Math. J. (2) 42 (1990), no. 4, 537–-560.

\bibitem[Ne]{Ne} Markus Neuhauser,
\it Kazhdan's property T for the symplectic group over a ring. 
\rm Bull. Belg. Math. Soc. Simon Stevin 10 (2003), no. 4, 537–-550.

\bibitem[Op1]{Op1} Izhar Oppenheim,
\it Vanishing of cohomology and property (T) for groups acting on weighted simplicial complexes,
\rm Groups Geom. Dyn. 9 (2015), no. 1, 67–-101.

\bibitem[Op2]{Op2} Izhar Oppenheim,
\it Averaged projections, angles between groups and strengthening of property $(T)$,
\rm  Math. Ann. 367 (2017), no. 1-2, 623–-666.

\bibitem[Re]{Re} Bertrand R\'emy,
\it Kac-Moody groups as discrete groups. 
\rm Essays in geometric group theory, 105–-124, Ramanujan Math. Soc. Lect. Notes Ser., 9, Ramanujan Math. Soc., Mysore, 2009.

\bibitem[Ro]{Ro} Guy Rousseau,
\it Groupes de Kac-Moody d\'eploy\'es sur un corps local II. Masures ordonn\'ees,
\rm Bull. Soc. Math. France, 144 (2016), 613--692.

\bibitem[Sh]{Shalom} Yehuda Shalom,
\it Bounded generation and Kazhdan's property $(T)$.
\rm Inst. Hautes \'Etudes Sci. Publ. Math. No. 90 (1999), 145--168.

 
\bibitem[Ti]{Ti} Jacques Tits,
\it Uniqueness and presentation of Kac-Moody groups over fields.
\rm J. Algebra  105  (1987),  no. 2, 542--573.
\end{thebibliography}
\end{document}